\documentclass[sts]{imsart}

\usepackage{booktabs}

\usepackage{geometry}
\usepackage{natbib}
\usepackage{enumitem}

\usepackage{amssymb,amsbsy,amsmath,amscd,amsthm,amsfonts,MnSymbol}
\usepackage[utf8]{inputenc}
\usepackage{hyperref}
\usepackage{dsfont}
\usepackage{graphicx}

\usepackage{tikz}
\usetikzlibrary{positioning}
\usetikzlibrary{calc}

\newtheorem{definition}{Definition}
\newtheorem{theorem}{Theorem}
\newtheorem{proposition}{Proposition}

\newtheorem{lemma}{Lemma}
\newtheorem{remark}{Remark}

\newtheorem{corollary}{Corollary}

\newtheorem{openquestion}{Open question}

\newcommand{\R}{\mathbb R}

\newcommand{\E}{\mathbb E}

\newcommand{\Sp}{\mathbb S}

\newcommand{\tr}{\mathrm{tr}}

\newcommand{\Log}{\mathrm{Log}}

\newcommand{\CAT}{\mathrm{CAT}}

\newcommand{\per}{\mathrm{per}}

\newcommand{\var}{\mathrm{var}}

\newcommand*\diff{\mathop{}\!\mathrm{d}}

\newcommand*\dist{\mathop{}\!\mathrm{d}}

\newcommand{\DS}{\displaystyle}

\begin{document}

\begin{frontmatter}

\title{Finite sample bounds for barycenter estimation in geodesic spaces}
\runtitle{Barycenter estimation}

\author{Victor-Emmanuel Brunel \thanks{CREST-ENSAE, victor.emmanuel.brunel@ensae.fr} and Jordan Serres \thanks{LPSM, Sorbonne, serres@lpsm.paris}}

\runauthor{V.-E. Brunel and J. Serres}

\setattribute{abstractname}{skip} {{\bf Abstract:} } 
\begin{abstract}
We study the problem of estimating the barycenter of a distribution given i.i.d. data in a geodesic space. Assuming an upper curvature bound in Alexandrov's sense and a support condition ensuring the strong geodesic convexity of the barycenter problem, we establish finite-sample error bounds in expectation and with high probability. Our results generalize Hoeffding- and Bernstein-type concentration inequalities from Euclidean to geodesic spaces. Building on these concentration inequalities, we derive statistical guarantees for two efficient algorithms for the computation of barycenters. 
\end{abstract}


\end{frontmatter}

\section{Introduction}

Statistics and machine learning are more and more confronted with data that lie in non-linear spaces. For instance, in spatial statistics (e.g., directional data), computational tomography (e.g., data in quotient spaces such as in shape statistics, collected up to rigid transformations), economics (e.g., optimal transport, where data are discrete measures), etc. Moreover, data that are encoded as very high dimensional vectors may have a much smaller intrinsic dimension, for instance, if they are lying on small dimensional submanifolds of the Euclidean space: In that case, leveraging the possibly non-linear geometry of the data can be a powerful tool in order to significantly reduce the dimensionality of the problem at hand, this phenomenon is understood as the \textit{manifold hypothesis}, which is extensively studied in the literature, see e.g. \citep{fefferman2016testing}. Even though more and more algorithms have been developed to work with such data \citep{LimPalfia14,ohtapalfia,zhang2016first,zhang2018towards}, there is still very little theoretical work for uncertainty quantification, especially in non-asymptotic regimes, which are pervasive in machine learning. In this work, we prove finite sample, high probability error bounds for barycenters of data points, which are the most natural extension of linear averaging to the context of non-linear geometries.


Let $(M,\dist)$ be a metric space. Given $x_1,\ldots,x_n\in M$ ($n\geq 1$), a barycenter of $x_1,\ldots,x_n$ is any minimizer of the function
\begin{equation} \label{eq:Frechet_emp}
    \frac{1}{n}\sum_{i=1}^{n}\dist(x_i,b)^2, \quad b\in M.
\end{equation} 
One can easily check that if $(M,\dist)$ is a Euclidean or Hilbert space, the minimizer is unique and it is given by the average of $x_1,\ldots,x_n$. More generally, given a probability distribution $\mu$ with two moments on $(M,\dist)$, one can define the barycenter of $\mu$ as the unique minimizer of the function 
\begin{equation} \label{eq:Frechet_pop}
    \int_M \dist(x,b)^2\diff\mu(x), \quad b\in M.
\end{equation}
Here, we say that $\mu$ has two moments if and only if the function $\dist(\cdot,b_0)^2$ is integrable with respect to $\mu$ for some $b_0\in M$ (and hence, by the triangle inequality, for all such $b_0$). Note that, in order to define a barycenter of $\mu$, it is in fact enough to assume that $\mu$ only has one finite moment, by subtracting $\dist(x,b_0)^2$ inside the integral of \eqref{eq:Frechet_pop}, for any fixed $b_0$ (one easily checks that the set of minimizers does not depend on the choice of $b_0$). However, in this work, we will always assume the existence of at least two moments, in order to obtain relevant statistical error bounds. 

The main question that we are concerned with is the following. Given a probability distribution $\mu$ on $(M,\dist)$ and $n$ independent, identically distributed (i.i.d) random points $X_1,\ldots,X_n$ with distribution $\mu$ ($n\geq 1$), how likely is a barycenter $\hat b_n$ of $X_1,\ldots,X_n$ (we call $\hat b_n$ an \textit{empirical barycenter}) to be far from a barycenter $b^*$ of $\mu$? In other words, we aim at bounding the statistical error $\dist(\hat b_n,b^*)$. Our focus will be on deriving high probability bounds that hold for any sample size $n\geq 1$. Moreover, our bounds will be dimension-free, i.e., they will not require the space $(M,\dist)$ to have finite dimension in any sense (e.g., doubling dimension).

Barycenters were initially introduced in statistics by \citep{Frechet48} in the 1940's, and later by \citep{karcher1977Riemannian}, where they were better known as Fréchet means, or Karcher means. They were popularized in the field of shape statistics \citep{kendall2009shape} and optimal transport \citep{agueh2011barycenters,cuturi2014fast,le2017existence,claici2018stochastic,kroshnin2019complexity,altschuler2021wasserstein,altschuler2022wasserstein} but also find applications in broader machine learning problems \citep{hu2023minimizing}.
The existence and uniqueness of barycenters are challenging problems in general \citep{afsari2011riemannian, Yokota16,Yokota18}. Asymptotic theory is well understood for empirical barycenters in various setups, particularly laws of large numbers \citep{ziezold1977expected} and central limit theorems in Riemannian manifolds -- a smooth structure on $M$ is a natural assumption in order to derive central limit theorems \citep{bhattacharya2003large,bhattacharya2005large,bhattacharya2017omnibus,eltzner2019smeary,eltzner2019stability}. Only very few non-asymptotic results have been proven so far, most of which hold under fairly technical conditions. \citep{sturm03} proposes an alternative definition of barycenters, which we will also review below, and obtains a bound on the expected statistical error when $(M,\dist)$ is \textit{non positively curved} (NPC) \citep[Theorem 4.7]{sturm03}. Namely, the bound reads as follows:
\begin{equation} \label{eq:Sturm}
    \E[\dist(\tilde b_n,b^*)]\leq \frac{\sigma^2}{n}
\end{equation}
where $\tilde b_n$ is the $n$-th iterated barycenter of $X_1,\ldots,X_n$ (we give its precise definition in Section~\ref{sec:barydef}) and $\sigma^2$ is the total variance of $\mu$, i.e., $\sigma^2=\E[d(X_1,b^*)^2]$. In Hilbert spaces, $\sigma^2$ coincides with the trace of the covariance operator. In particular, \eqref{eq:Sturm} is sharp in the sense that it is in fact an equality when $(M,\dist)$ is a Hilbert space.
Much later, \citep[Corollary 11]{fastconv} provides the same inequality for $\hat b_n$, under the extra constraint that $(M,\dist )$ has curvature bounded from below. At a high level, this means that the space $(M,\dist )$ does not exhibit branching (i.e., a geodesic cannot split, unlike, for instance, in metric trees) and this ensures some regularity of the tangent cones of $M$, allowing to perform local linearizations. They also extended their result to spaces $(M,\dist )$ that may have positive curvature, so long as they satisfy a so-called hugging condition. However, except for NPC spaces, there is no explicit example that satisfies such a condition. 

In a recent work, \citep{escande2023concentration} proves that the same upper bound as \citep[Corollary 11]{fastconv}, up to an additional multiplicative factor, holds in any NPC space, dropping the curvature lower bound assumption, by elegantly leveraging the quadruple inequality, which characterizes NPC spaces \citep[Corollary 3]{berg2008quasilinearization}.
Several non-asymptotic, high probability bounds have also been established for empirical and iterated barycenters. \citep[Eq. (3.10)]{fastconv} proposes a definition of sub-Gaussian random variables, closely related to the one we give below. Under the hugging condition mentioned above, they prove (Theorem 12) a nearly sub-Gaussian tail bound the empirical barycenter of i.i.d sub-Gaussian random variables, with a residual term that decays exponentially fast with $n$. \citep{convrate} obtains concentration inequalities for the empirical barycenter $\hat b_n$ of i.i.d, bounded random variables with non-parametric rates, under some metric entropy conditions on $(M,\dist )$, some of which being similar in spirit to requiring $M$ to have finite dimension. 
\citep{Funano10} establishes a high probability bound in NPC spaces for the iterated barycenter $\tilde b_n$ of i.i.d bounded random variables, by assuming that $M$ is either a metric tree or a finite dimensional Riemannian manifold. In the latter case, the bound derived by \citep{Funano10} depends on the dimension of $M$ and, hence, does not extend to infinite dimensional spaces.  

\subsection{Our contributions}

We prove error bounds in expectation and with high probability for barycenter estimation in geodesic spaces that have a curvature upper bound. Contrary to certain works in the previous literature, we do not assume any curvature lower bounds, which allows us to cover broader classes of spaces such as branching spaces: Metric trees, BHV space of phylogenetic trees \citep{billera2001geometry}, and spaces of non-constant dimension such as stratified spaces. 

The barycenter estimators that we treat are: \textit{Empirical barycenters}, defined as barycenters of the data points, and \textit{iterated barycenters}, which are the output of a stochastic proximal descent algorithm with appropriate step-sizes. While first order stochastic algorithms (namely, stochastic gradient descent) have been studied extensively in Riemmanian manifolds \citep{bonnabel2013stochastic,zhang2016first} with non-asymptotic guarantees holding in expectation, our results provide new bounds for iterated barycenters: Dimension-free, high probability bounds in spaces of non-positive curvature, and bounds in expectation that are independent of the size of the domain of the data, unlike the general bounds given in \citep{zhang2016first}.

Our bounds are always dimension free, in the sense that they do not require the space to have finite dimension (in any sense, e.g., Hausdorff dimension, or doubling dimension): They involve features of the data distribution (e.g., sub-Gaussian norm and total variance) but not the dimension of the underlying space. In particular, our high probability bounds extend the classical Hoeffding's and Bernstein's inequalities to a non-linear setup.

This work is an extension of our conference paper \citep{brunel2024concentration}, which only applied to NPC spaces. Perhaps surprisingly, extending the results of \citep{brunel2024concentration} to geodesic spaces with curvature bounded above by any $\kappa\in\R$ required significant effort and has opened new questions, which we pose as \textit{open questions} in this paper. 

\subsection{Outline}

Our work is organized as follows. In Section~\ref{sec:CATspaces}, we first give a brief introduction to geodesic metric spaces with curvature upper bounds and to the notion of geodesic convexity in such spaces. Then, we define barycenter estimators, which we analyze from a geometric point of view. In Section~\ref{sec:Laplace}, we develop the main tools that allow us to obtain measure concentration of functions of random points in metric spaces. Finally, our main statistical results on barycenter estimation are stated in Section~\ref{sec:main}. Some proofs are deferred to the appendix.

\subsection{Notation and general definitions}

Let $(M,\dist)$ be a metric space. For all $x_0\in M$ and $r\geq 0$, we denote by $B(x_0,r)$ the closed ball centered at $x_0$ and with radius $r$. The diameter of a bounded subset $B$ of $M$ is denoted by $\textrm{diam}(B)$.

For any $x,y\in M$, we call a (constant speed) geodesic from $x$ and $y$ any path $\gamma:[0,1]\to M$ satisfying $\gamma(0)=x, \gamma(1)=y$ and $\dist(\gamma(s),\gamma(t))=|s-t|\dist(x,y)$ for all $s,t\in[0,1]$. The set of geodesics from $x$ to $y$ is denoted by $\Gamma_{x,y}$ and we say that $(M,\dist)$ is a geodesic space if $\Gamma_{x,y}$ is non-empty for every pair of points $x,y\in M$. Note that $\Gamma_{x,y}$ might not be a singleton, for instance when $M$ is a Euclidean sphere equipped with its geodesic distance and $x$ and $y$ are antipodal. We call a geodesic ray any map $\eta:[0,\infty)\to M$ satisfying $\dist(\eta(s),\eta(t))=|s-t|$ for all $s,t\geq 0$ (note that geodesic rays may not exist, e.g., if $M$ is bounded). 

A random variable $X$ (resp. a probability measure $\mu$) in $(M,\dist)$ is said to have $k$ moments, $k\geq 1$, if $\E[\dist(X,x_0)^k]<\infty$ (resp. $\int_M \dist(x,x_0)^k\diff\mu(x) <\infty$) for some $x_0\in M$. By the triangle inequality, if this holds for some $x_0\in M$, then it must hold for all such $x_0\in M$. If $X$ has two moments, we define its total variance as $\inf_{x\in M} \E[\dist(X,x)^2]$.

For a random variable $X$ (resp. a probability measure $\mu$) in $(M,\dist)$ with two moments, a barycenter of $X$ (resp. $\mu$) is any minimizer $b\in M$ of $\E[\dist(b,X)^2]$ (resp. $\int_\mu \dist(b,x)^2\diff\mu(x)$). As mentioned above, one could define barycenters of random variables or distributions with only one moment, but in this work, we will always assume the existence of at least two moments. The map $b\in M\mapsto \E[\dist(b,X)^2]$ is called the \textit{Fréchet function} of $X$, or of the distribution of $X$. For instance, if $(M,\dist)$ is a geodesic space, one can verify that for all pairs $x,y\in M$ and for all $t\in [0,1]$, any point of the form $\gamma(t), \gamma\in\Gamma_{x,y}$, is a barycenter of $(1-t)\delta_x+t\delta_y$ (aka weighted barycenter of $x$ and $y$, with respective weights $1-t$ and $t$). When $t=1/2$, we simply obtain a midpoint of $x$ and $y$.

For any $x\in M$, we denote by $\delta_{x}$ the Dirac measure at $x$, i.e., $\delta_x(A)=\mathds 1_{x\in A}$ for all Borel sets $A\subseteq M$.

Many constants are used throughout the paper, they are summarized in Table \ref{tab:constants}.

\begin{table}[t]
\centering
\caption{Summary of the main constants used throughout the paper.}
\label{tab:constants}
\footnotesize
\begin{tabular}{@{}lp{4.3cm}lll@{}}
\toprule
Symbol & Role & $\kappa\le 0$ & $\kappa>0$ & Source \\
\midrule
$D_\kappa$ & Diameter of the model space $M_\kappa$ & $\infty$ & $\pi/\sqrt\kappa$ & Sec.~\ref{sec:CATspaces} \\[4pt]
$\alpha(\varepsilon,\kappa)$ & Strong-convexity of $\dist(\cdot,x_0)^2$ & $2$ & $(\pi-2\sqrt\kappa\,\varepsilon)\tan(\varepsilon\sqrt\kappa)$ & Lem.~\ref{lem:d2strconvex} \\[6pt]
$L$ & $\hat B_n$ is $L/n$-Lipschitz on $C^n$ & $1$ &  $\frac{\pi}{\alpha(\varepsilon,\kappa)}$ & Thm.~\ref{thm:Lip_Bary} \\[4pt]
$A$ & $\E[\dist(\hat b_n,b^*)^2]\le A\,\sigma^2/n$ (i.i.d.) & $2$ & $\dfrac{32}{\varepsilon^{1/4}\kappa^{1/8}\,\alpha(\varepsilon,\kappa)}$ & Thm.~\ref{thm:llnempiricalbar} \\[8pt]
$\tilde A$ & $\E[\dist(\hat b_n,b^*)^2]\le \tilde A\,\bar\sigma_n^2/n$ (heteroskedastic) & $2$ & $\dfrac{{64}\sqrt2}{\varepsilon^{1/4}\kappa^{1/8}\,\alpha(\varepsilon,\kappa)}$ & Thm.~\ref{thm:hetero1} \\[8pt]
$t_k$ & Step sizes of the iterated barycenter $\tilde b_n$ & $1/k$ & $\dfrac{2}{\alpha(\varepsilon,\kappa)\,k+2}$ & Thm.~\ref{thm:llnvariance} \\
\bottomrule
\end{tabular}
\end{table}

\section{CAT spaces and convex domains} \label{sec:CATspaces}

\subsection{Model spaces and curvature bounds}

Here, we only briefly recall the definition of CAT spaces, i.e., metric spaces with global curvature upper bounds in Alexandrov's sense. For more details, we refer the reader to \citep{burago2022course,Alexandrovgeom} and \citep{bridson2013metric}.
First, we recall the definition of model spaces of constant curvature. Fix $\kappa\in\R$.
\begin{itemize}
    \item \underline{$\kappa=0$:} Euclidean plane. Set $M_0=\R^2$ equipped with its Euclidean metric. This is a geodesic space where geodesics are unique and given by line segments. 

    \item \underline{$\kappa>0$:} Sphere. Set $M_\kappa=\frac{1}{\sqrt\kappa}\Sp^2$: This is the $2$-dimensional Euclidean sphere, embedded in $\R^3$, with center $0$ and radius $1/\sqrt\kappa$, equipped with the arc length metric defined as $\dist_\kappa(x,y)=\frac{1}{\sqrt\kappa}\arccos(\kappa x^\top y)$, for all $x,y\in M_\kappa$. This is a geodesic space where the geodesics are unique except for antipodal points, and given by arcs of great circles. Here, a great circle is the intersection of the sphere with any plane going through the origin in $\R^3$. 

    \item \underline{$\kappa<0$:} Hyperbolic space. Set $M_\kappa=\frac{1}{\sqrt{-\kappa}}\mathbb H^2$, where $\mathbb H^2=\{(x_1,x_2,x_3)\in\R^3:x_3>0, x_1^2+x_2^2-x_3^2=-1\}$. The metric is given by $\dist_\kappa(x,y)=\frac{1}{\sqrt{-\kappa}}\textrm{arccosh}(-\kappa\langle x,y\rangle)$, for all $x,y\in M_\kappa$, where $\langle x,y\rangle=x_1y_1+x_2y_2-x_3y_3$. This is a geodesic space where geodesics are always unique and are given by the intersections of $M_\kappa$ with planes going through the origin in $\R^3$. 
\end{itemize}

Let $\DS D_\kappa=\begin{cases} \infty \mbox{ if } \kappa\leq 0 \\ \frac{\pi}{\sqrt{\kappa}} \mbox{ if } \kappa>0\end{cases}$ be the diameter of the model space $M_\kappa$. A fundamental property of $M_\kappa$ is that between any two points $x,y\in M_\kappa$ with $\dist_\kappa(x,y)<D_\kappa$, there is a unique geodesic, i.e., $\Gamma_{x,y}$ is always a singleton unless $\kappa>0$ and $x$ and $y$ are antipodal points on the sphere $M_\kappa$.
The notion of curvature (lower or upper) bounds for a geodesic metric space $(M,\dist )$ is defined by comparing the triangles in $M$ with their counterparts in model spaces. 

\begin{definition}
	A (geodesic) triangle in $M$ is a set of three points in $M$ (the vertices) together with three geodesics connecting them (the sides). 
\end{definition}

Given three points $x,y,z\in M$, we abusively denote by $\Delta(x,y,z)$ a triangle with vertices $x,y,z$, with no mention to which geodesics are chosen for the sides, which are not necessarily unique. The perimeter of a triangle $\Delta=\Delta(x,y,z)$ is defined as $\per(\Delta)=\dist(x,y)+\dist(y,z)+\dist(x,z)$. It does not depend on the choice of the sides. 

\begin{definition}
	Let $\kappa\in\R$ and $\Delta$ be a triangle in $M$ with $\per(\Delta)<2D_\kappa$. A comparison triangle for $\Delta$ in the model space $M_\kappa$ is a triangle $\bar\Delta\subseteq M_\kappa$ with same side lengths as $\Delta$, i.e., if $\Delta=\Delta(x,y,z)$, then $\bar\Delta=\Delta(\bar x,\bar y,\bar z)$ where $\bar x,\bar y,\bar z$ are points in $M_\kappa$ satisfying 
	$$\begin{cases} \dist(x,y)=\dist_\kappa(\bar x,\bar y) \\ \dist(y,z)=\dist_\kappa(\bar y,\bar z) \\ \dist(x,z)=\dist_\kappa(\bar x,\bar z). \end{cases}$$
\end{definition}

Note that in $M_\kappa$, any side of a triangle with perimeter less than $2D_\kappa$ must be of length less than $D_\kappa$, so the geodesics connecting the vertices are unique. Moreover, given any positive numbers $a,b,c$ with $a\leq b+c$, $b\leq a+c$, $c\leq a+b$ and $a+b+c<2D_\kappa$, there exists a unique triangle in $M_\kappa$ with side lengths given by $a,b$ and $c$, up to rigid transformations. Therefore, comparison triangles are always unique up to isometries of the model space $M_\kappa$.
We are now ready to define curvature bounds. Intuitively, we say that $(M,\dist)$ has global curvature bounded from above by $\kappa$ if all its triangles with perimeter smaller than $2D_\kappa$ are thinner than their comparison triangles in the model space $M_\kappa$. 

\begin{definition} \label{def:CAT}
	Let $(M,\dist)$ be a metric space and $\kappa\in\R$. 
 \begin{itemize}
    \item We say that $(M,\dist )$ has global curvature bounded from above by $\kappa$ if and only if for all triangles $\Delta\subseteq M$ with $\per(\Delta)<2D_\kappa$ and for all $x,y\in \Delta$, $\dist(x,y)\leq \dist_\kappa(\bar x,\bar y)$, where $\bar x$ and $\bar y$ are the points on a comparison triangle $\bar\Delta$ in $M_\kappa$ that correspond to $x$ and $y$ respectively. 
    \item We say that $(M,\dist)$ is a $\CAT(\kappa)$ space if it is a geodesic space, complete (in the topological sense) and has global curvature bounded from above by $\kappa$. 
    \item We say that $(M,\dist)$ is a $\CAT$ space if it is a $\CAT(\kappa)$ space for some $\kappa\in\R$.
 \end{itemize}
\end{definition}
 Let us mention two natural properties of $CAT$ spaces. First, for all $\kappa,\kappa'\in\R$ with $\kappa\leq \kappa'$, it holds that any $\CAT(\kappa)$ space is also a $\CAT(\kappa')$ space. Second, if $(M,d)$ is a $CAT(\kappa)$ space, then the $\rho$-dilation $(M,\rho\,d)$, $\rho>0$, is a $CAT(\kappa/\rho^2)$ space.
For instance, it is obvious that any Euclidean or Hilbert space is a $\CAT(0)$ space and that a Euclidean sphere with radius $r>0$ is a $\CAT(\kappa)$ space with $\kappa=1/r^2$. A Riemannian manifold that is simply connected and has sectional curvature uniformly bounded from above by $\kappa\in\R$ is a $\CAT(\kappa)$ space (see \citep{burago2022course}). Here is a list of more specific examples.

\begin{itemize}
    \item Any metric tree is a $\CAT(0)$ space. A metric tree is a complete metric space $(M,\dist)$ where for all $x,y,z\in M$, there exists some $w\in M$ with $\dist(x,y)=\dist(x,w)+\dist(w,y)$, $\dist(x,z)=\dist(x,w)+\dist(w,z)$ and $\dist(y,z)=\dist(y,w)+\dist(w,z)$. For instance, any acyclic graph with positive edge weights can be equipped with a metric that makes it a metric tree where the length of each edge coincides with its weight. Metric trees are simple enough non-Euclidean $\CAT(0)$ spaces that can provide intuition and/or counterexamples.

    \item The space $\mathcal S_d^{++}$ of $d\times d$ symmetric positive definite matrices can be equipped with several different metrics, making it (or portions of it) a $\CAT(\kappa)$ space for different values of $\kappa$. For instance, the Euclidean metric $\dist_1(A,B)=\|B-A\|_{\textsf{F}},\,\, A,B\in\mathcal S_d^{++}$, makes it a $\CAT(0)$ space (here, $\|\cdot\|_{\textsf{F}}$ is the Fröbenius norm). The metric $\dist_2(A,B)=\|\log(A^{-1/2}BA^{-1/2})\|_{\textsf{F}}$ also makes it $\CAT(0)$. In fact, it can be seen that this metric is inherited from a Riemannian structure and that midpoints, with respect to this metric, are given by geometric means. That is, given any $A,B\in \mathcal S_d^{++}$, there exists a unique geodesic from $A$ to $B$ and its midpoint is $A^{1/2}(A^{-1/2}BA^{-1/2})^{1/2}A^{1/2}$, the geometric mean of $A$ and $B$. See \citep{bhatia2006riemannian} for a more detailed account on operator geometric mean and this Riemannian structure. Finally, a third metric that we mention here is the Bures-Wasserstein metric, which is also inherited from a Riemannian structure. This metric comes from optimal transport and can be defined as follows. Given any $A,B\in\mathcal S_d^{++}$, define $\dist_3(A,B)$ as the Wasserstein $2$ distance between $\mathcal N_d(0,A)$ and $\mathcal N_d(0,B)$, the $d$-variate centered Gaussian distributions with respective covariance matrices $A$ and $B$. It can be shown that $\dist_3(A,B)=\min\{\|M-N\|_{\textsf{F}}:M,N\in \R^{d\times d}, MM^\top=A, NN^\top=B\}=\min_{U\in \mathcal O(d)} \|A^{1/2}-UB^{1/2}\|_{\textsf{F}}$, where $\mathcal O(d)$ is the set of $d\times d$ orthogonal matrices. Then, for all $\lambda>0$, the collection of all $A\in\mathcal S_d^{++}$ with all eigenvalues at least $\lambda$ is a $\CAT(\kappa)$ space with $\kappa=3/(2\lambda^2)$ \citep[Proposition 2]{massart2019curvature}. 
\end{itemize}
An important fact about $\CAT$ spaces is that geodesics between points that are close enough are always unique. We state this as a proposition, whose proof can be found in \citep[Section 9.8]{alexander2024alexandrov}.

\begin{proposition} \label{prop:uniquegeodesics}
    Let $(M,\dist)$ be a $\CAT(\kappa)$ space for some $\kappa\in\R$. Then, for all $x,y\in M$ with $\dist(x,y)<D_\kappa$, there is a unique geodesic from $x$ to $y$.
\end{proposition}

The remaining definitions and properties presented in this section will only be useful for the proof of Theorem~\ref{thm:Lip_Bary} below, which is central in this work.

\begin{definition}[Alexandrov angles]
    Let $(M,\dist)$ be a geodesic space and $x,y,z\in M$ with $x\neq z, y\neq z$. Let $\gamma\in\Gamma_{z,x}$ and $\gamma'\in\Gamma_{z,y}$. For $t,t'\in[0,1]$, let $\bar x_t,\bar y_{t'}\in\R^2$ be the vertices of a comparison triangle $\Delta(\bar x_t,\bar y_{t'},\bar z)$ for the triangle $\Delta(\gamma(t),\gamma'(t'),z)$.
    The Alexandrov angle between $\gamma_1$ and $\gamma_2$ is defined as $$\angle_z(\gamma_1,\gamma_2)=\limsup_{t,t'\to 0} \overline\angle_{\bar z}(\bar x_t,\bar y_{t'})$$
    where $\overline\angle_{\bar z}(\bar x_t,\bar y_{t'})\in [0,\pi]$ is the angle at vertex $\bar z$ in the Euclidean triangle $\Delta(\bar x_t,\bar y_{t'},\bar z)$. When convenient, we also denote this Alexandrov angle by $\angle_z(x,y)$. 
\end{definition}

\begin{definition}[Tangent cone]
    Let $(M,\dist)$ be a geodesic space and $z\in M$. 
    \begin{itemize}
        \item For $x,y\in M$ with $x\neq z$ and $y\neq z$ and $\gamma_1\in\Gamma_{z,x}$ and $\gamma_2\in\Gamma_{z,y}$, we say that $\gamma_1$ and $\gamma_2$ are equivalent if and only if $\angle_z(\gamma_1,\gamma_2)=0$.
        \item For $\gamma\in\Gamma_{z,x}$, for any $x\in M$ with $x\neq z$, denote by $\bar\gamma$ the equivalence class of $\gamma$ with respect to the equivalence relation described above. Let $\tilde T_zM$ be the collection of all pairs of the form $(t,\bar\gamma)$ for $t\geq 0$ and $\gamma\in\Gamma_{z,x}$ for some $x\in M\setminus\{z\}$, where any two pairs of the form $(0,\bar\gamma),(0,\bar\gamma')$ are identified. Equip it with the distance $\dist_{T_{z}M}((s,\bar\gamma_1),(t,\bar\gamma_2))=\left(s^2+t^2-2st\cos(\angle_z(\gamma_1,\gamma_2))\right)^{1/2}$ for any choice of $\gamma_1\in\bar\gamma_1$ and $\gamma_2\in\bar\gamma_2$. The completion of this space is called the tangent cone to $M$ at $z$ and it is denoted by $T_zM$.
    \end{itemize}
\end{definition}

Validity of these definitions can be checked in \citep[Section II.3]{bridson2013metric} (note that in that reference, the terminology ``tangent cone" refers to $\tilde T_{z}M$ instead of its completion). The following property will be used later.

\begin{proposition}[{\citep[Theorem II.3.19]{bridson2013metric}}]
    With the notation above, if $M$ is a $\CAT(\kappa)$ space for some $\kappa\in\R$, $T_zM$ is a $\CAT(0)$ space. 
\end{proposition}

Now, let us introduce the following, standard notation, coming from Riemannian geometry. Let $(M, \dist)$ be a geodesic space and $z\in M$. Let $x\in M$ such that there is a unique geodesic $\gamma\in\Gamma_{z,x}$. Then, we denote by $\Log_z(x)=(\dist(z,x),\bar\gamma)$. It is called logarithm because in the smooth case, it corresponds to the inverse of the Riemannian exponential map. Then, we have the following lemma.

\begin{lemma} \label{lem:comp_dist}
    Let $(M,\dist)$ be a $\CAT(\kappa)$ space for some $\kappa>0$. For all points $x,y,z\in M$ satisfying $\max(\dist(x,z),\dist(y,z))\leq r < \pi/\sqrt\kappa$,
    $$\dist_{T_zM}(\Log_z(x),\Log_z(y))\leq \frac{\sqrt\kappa r}{\sin(\sqrt\kappa r)}\dist(x,y).$$
\end{lemma}

\begin{proof}
    By definition, letting $s=\dist(x,z), t=\dist(y,z)$ and $\theta=\angle_z(x,y)$, the angle comparison theorem \citep[Proposition II.3.1]{bridson2013metric} yields that $\theta\geq \angle_{\bar z}^{(\kappa)}(\bar x,\bar y)$ where $\bar x,\bar y$ and $\bar z$ are the vertices of a comparison triangle of $\Delta(x,y,z)$ in $M_\kappa$. Since $\cos$ is non-increasing on $[0,\pi]$, we therefore obtain that 
    \begin{align}
        \dist_{T_zM}(\Log_z(x),\Log_z(y))^2 & = s^2+t^2-2st\cos(\theta) \nonumber \\
        & \leq s^2+t^2-2st\cos(\angle_{\bar z}^{(\kappa)}(\bar x,\bar y)) \nonumber \\
        & = \dist_{T_{\bar z}M_\kappa}(\Log_{\bar z}(\bar x),\Log_{\bar z}(\bar y))^2.
    \end{align}

The conclusion is now straightforward using \citep[Lemma 10.13 and Theorem 10.14]{lee2018introduction} for the Euclidean sphere of constant curvature $\kappa$, since for all $\rho\in [0,r]$, $1\geq \frac{\sin(\sqrt\kappa \rho)}{\sqrt\kappa \rho}\geq \frac{\sin(\sqrt\kappa r)}{\sqrt\kappa r}$.

\end{proof}

Finally, let us define Busemann functions, which will be used later in Section \ref{sec:barydef}.

\begin{definition} \label{def:Busemann}
    Let $\eta$ be a geodesic ray in a metric space $(M,\dist)$. The Busemann function associated with $\eta$ is the map $\mathfrak b_\eta:M\to\R$ defined by setting
    $$\mathfrak b_\eta(x)=\lim_{t\to\infty} \dist(x,\eta(t))-t$$
    for all $x\in M$. 
\end{definition}

The existence of the limit in the above definition is guaranteed by \citep[Lemma II.8.18]{bridson2013metric}. Moreover, it can be checked that Busemann functions are always $1$-Lipschitz (by the reverse triangle inequality, applied for each $t\geq 0$ separately). Busemann functions play the role of one-dimensional projections: In Euclidean space $\R^d$, if $\eta(t)=x_0+tu, t\geq 0$, where $x_0\in\R^d$ and $u$ is a unit vector, then $\mathfrak b_\eta(x)=u^\top (x-x_0)$ for all $x\in\R^d$.

\subsection{Convexity in metric spaces}

Let $(M,\dist)$ be a metric space. A subset $A\subseteq M$ is called (geodesically) convex if and only if for all $x,y\in A$ and all $\gamma\in\Gamma_{x,y}$, $\gamma([0,1])\subseteq A$. A function $f:A\to\R$ defined on a convex subset $A$ of $M$ is called (geodesically) convex (on $A$) if and only if it is convex along all geodesics, i.e., for all $x,y\in M$, $\gamma\in\Gamma_{x,y}$ and $t\in [0,1]$, it holds that $f(\gamma(t))\leq (1-t)f(x)+tf(y)$. The function $f$ is called $\alpha$-strongly (geodesically) convex, for $\alpha>0$, if and only if for all $x,y\in M$, $\gamma\in\Gamma_{x,y}$ and $t\in [0,1]$, it holds that $f(\gamma(t))\leq (1-t)f(x)+tf(y)-\frac{\alpha}{2}t(1-t)\dist(x,y)^2$.
Here, we give some basic yet useful facts related to convexity in metric spaces. The first one concerns the convexity of the squared distance to a given point. Recall that $D_\kappa=\pi/\sqrt\kappa$ for all $\kappa>0$ and $D_\kappa=\infty$ for all $\kappa\leq 0$.

\begin{lemma} \label{lem:d2strconvex}
    Let $\kappa\in\R$ and $(M,\dist)$ be a $\CAT(\kappa)$ space. The following properties hold true.
    \begin{itemize}
        \item All balls of radius less than $D_\kappa/2$ are convex.
        \item If $\kappa\leq 0$ then $\dist(x_0,\cdot)^2$ is $2$-strongly convex for all choices of $x_0\in M$.
        \item If $\kappa>0$, then for all $\varepsilon>0$ and all $x_0\in M$, $\dist(x_0,\cdot)^2$ is $\alpha(\varepsilon,\kappa)$-strongly convex on the ball $B\left(x_0,D_\kappa/2-\varepsilon\right)$, with $\alpha(\varepsilon,\kappa)=(\pi-2\sqrt{\kappa}\varepsilon)\tan(\varepsilon\sqrt{\kappa})$. In particular, for all $\varepsilon>0$ and all balls $B$ of radius at most $1/2(D_\kappa/2-\varepsilon)$, $\dist(x_0,\cdot)^2$ is $\alpha(\varepsilon,\kappa)$-strongly convex on $B$ for all choices of $x_0\in B$.
    \end{itemize} 
\end{lemma}

The first part of this lemma states that the squared distance to a given point is always $2$-strongly convex in a $\CAT(\kappa)$-space for any $\kappa\leq 0$, just as in Euclidean or Hilbert spaces. This is proved in \citep[Proposition 2.3]{sturm03}. The strong convexity constant $2$ cannot be improved even for negative $\kappa$, since $\dist(\cdot,x_0)^2$ is exactly $2$-strongly convex along any geodesic going through $x_0$. The case of positive $\kappa$ is proved in \citep[Proposition 3.1]{Ohtaconvexity}.
In the sequel, $\alpha(\varepsilon,\kappa)$ is as defined in the above lemma. The function $\alpha$ is decreasing as $\varepsilon\sqrt\kappa$ increases from $0$ to $\pi/2$ and satisfies $\alpha(\varepsilon,\kappa)\in (0,2)$. Moreover, it vanishes as $\varepsilon$ goes to zero for a fixed $\kappa>0$ and it goes to $2$ as $\varepsilon$ goes to $\frac{\pi}{2\sqrt{\kappa}}$ for a fixed $\kappa>0$. In fact, one has the following inequality, for all $\kappa>0$ and $\varepsilon\in (0,\pi/(2\sqrt\kappa))$, 
\begin{equation} \label{eq:alpha}
    \frac{4}{\pi}\varepsilon\sqrt\kappa\leq \alpha(\varepsilon,\kappa)\leq \pi\varepsilon\sqrt\kappa.
\end{equation}

\begin{definition}
    Let $(M,\dist)$ be a $\CAT(\kappa)$ space form some $\kappa\in\R$. A \textit{convex domain} is:
    \begin{itemize}
        \item Any closed convex subset of $M$ if $\kappa\leq 0$.
        \item Any closed, convex subset of $M$ that is included in some (closed) ball of radius less than $D_\kappa/4$ if $\kappa>0$.
    \end{itemize}
\end{definition}

According to our definition, a convex domain always is a convex subset of $M$, but the converse is not true when $\kappa>0$. However, note that in the model space $M_\kappa$ for $\kappa>0$, the only convex, closed ball of radius larger or equal to $D_\kappa/2$ is $M_\kappa$ itself. The reason of this discrepancy is that thanks to Lemma~\ref{lem:d2strconvex}, if $C$ is a convex domain, then $\dist(x_0,\cdot)^2$ is strongly convex on $C$ {\bf for all $x_0\in C$} and that would not necessarily be the case if $C$ was, say, a ball of radius larger than $D_\kappa/4$.

The following lemma appears in \citep[Proposition II.2.4]{bridson2013metric} for $\kappa\leq 0$ and in \citep[Proposition 3.5]{espinola2009cat} for $\kappa=1$ (and hence, via rescaling the metric $\dist$, for any $\kappa>0$). 

\begin{lemma}[Metric projection onto a convex domain] \label{lem:metricproj}
    Let $M$ be a $\CAT(\kappa)$ space and $C$ be a convex domain in $M$. If $\kappa>0$, let $B$ be a ball of radius less than $D_\kappa/4$ containing $C$. Otherwise, set $B=M$. Then, for all $x\in B$, there is a unique $y\in C$ satisfying $\dist(x,y)=\dist(x,C)=\inf_{z\in C}\dist(x,z)$. Moreover, $y$ satisfies 
    $$\dist(x,z)>\dist(y,z), \quad \forall z\in C\setminus\{y\}.$$
    The point $y$ is called the metric projection of $x$ onto $C$.
\end{lemma}

Finally, as a consequence of Lemma~\ref{lem:d2strconvex}, the Fréchet function $F=\E[\dist(\cdot,X)^2]$ associated with a random variable $X$ with two moments and supported in a convex domain of a $\CAT$ space is strongly convex, as a convex combination of strongly convex functions. The following lemma will allow to establish essential properties on Fréchet functions and barycenters. 

\begin{lemma} \label{lem:quadr_growth}
    Let $(M,\dist)$ be a geodesic space and $C\subseteq M$ be a convex set. Let $f:C\to\R$ be a function that is $\alpha$-strongly convex, for some $\alpha\in\R$. Further assume that $f$ has a minimizer $x^*\in C$. Then, for all $x\in C$, 
    $$f(x)\geq f(x^*)+\frac{\alpha}{2}\dist(x,x^*)^2.$$
\end{lemma}

\begin{proof}
    Let $x\in C$ and let $\gamma\in\Gamma_{x^*,x}$. Then, for all $t\in (0,1)$, $f(x^*)\leq f(\gamma(t))\leq (1-t)f(x^*)+tf(x)-\frac{\alpha}{2}t(1-t)\dist(x,x^*)^2$. The result follows by rearranging, dividing by $t$ and letting $t\to 0$.
\end{proof}

From Lemmas~\ref{lem:metricproj} and \ref{lem:quadr_growth}, we obtain the following result.

\begin{proposition}[Variance inequality] \label{lem:varianceineq}
Let $(M,\dist)$ be a $\CAT(\kappa)$ space with $\kappa\in\R$ and let $X$ be a random variable in $M$ with two moments and supported in a convex domain $C\subseteq M$. Then, $X$ has a unique barycenter $b^*$. Moreover, $b^*\in C$ and one has the following variance inequality:
\begin{equation*}
    \frac{\alpha}{2}\dist(x,b^*)^2 \leq \E\left[\dist(x,X)^2-\dist(b^*,X)^2\right], \quad \forall x\in C
\end{equation*}
where $\alpha=2$ if $\kappa\leq 0$ and $\alpha=\alpha(\varepsilon,\kappa)$ if $\kappa>0$, where $\varepsilon>0$ is such that $C$ is contained in some ball of radius $1/2(D_\kappa/2-\varepsilon)$.
\end{proposition}

The proof of this lemma is covered in \citep[Propositions 4.3 and 4.4]{sturm03} when $\kappa\leq 0$. Hence, we only focus on the case when $\kappa>0$.

\begin{proof}
Suppose $\kappa>0$ and let $B=B(x_0,1/2(D_\kappa/2-\varepsilon))$ containing $C$, for some $x_0\in M$. Denote by $F(x)=\E[\dist(x,X)^2], x\in M$, the Fréchet function associated with $X$. 
Existence and uniqueness of the minimizer $b^*$ of $F$ on $M$, together with the fact that $b^*\in B$, are proved in \citep[Theorem B]{Yokota16}, using completeness of the space together with the strong convexity of the Fréchet function $F$ on $B$. Let $\tilde b^*$ be the metric projection of $b^*$ onto $C$. Lemma~\ref{lem:metricproj} yields that $F(\tilde b^*)\leq F(b^*)$, so it must hold that $\tilde b^*=b^*$, hence, $b^*\in C$.
Now, the variance inequality follows directly from Lemma~\ref{lem:quadr_growth}, since $F$ is $\alpha(\varepsilon,\kappa)$-strongly convex on $C$.
\end{proof}

Proposition~\ref{lem:varianceineq} also applies to the barycenter of any finite collection of points in a convex domain: Given a convex domain $C$ of a $\CAT(\kappa)$ space $(M,\dist)$ and $x_1,\ldots,x_n\in C$ ($n\geq 1$), applying Proposition~\ref{lem:varianceineq} to the distribution $n^{-1}\sum_{i=1}^n\delta_{x_i}$ yields that $x_1,\ldots,x_n$ have a unique barycenter $b_n$, that $b_n\in C$ and that  
$$\frac{1}{n}\sum_{i=1}^n (\dist(x_i,x)^2-\dist(x_i,b_n)^2)\geq \frac{\alpha}{2}\dist(x,b_n)^2, \, \forall x\in C$$
where $\alpha=2$ if $\kappa\leq 0$ and $\alpha=\alpha(\varepsilon,\kappa)$ if $\kappa>0$ and $C$ is included in a ball of radius $1/2(D_\kappa/2-\varepsilon)$.

\subsection{Barycenter functions} \label{sec:barydef}

Let $\kappa\in\R$, $(M,\dist)$ be a $\CAT(\kappa)$ space and $C$ be a convex domain of $M$. Let $n\geq 1$ be a fixed integer and $x_1,\ldots,x_n\in C$. By Proposition~\ref{lem:varianceineq}, the distribution $n^{-1}\sum_{i=1}^n\delta_{x_i}$ has a unique barycenter, which belongs to $C$. We denote it by $\hat B_n(x_1,\ldots,x_n)$. 
In the sequel, we denote by $\dist_1^{(n)}$ the $\ell^1$-product distance on $M^n$, which is given by $\dist_1^{(n)}((x_1,\ldots,x_n),(y_1,\ldots,y_n))=\sum_{i=1}^n \dist(x_i,y_i)$. The following theorem provides a sensitivity analysis of the barycenter function $\hat B_n$ with respect to this metric. Let us mention that a similar result was obtained in \citep[Theorem 2 and Lemma 1]{reimherr2021differential} in the case of Riemannian manifolds. 

\begin{theorem} \label{thm:Lip_Bary}
Let $(M,\dist)$ be a $\CAT(\kappa)$ space for some $\kappa\in\R$ and let $C$ be a convex domain. Then, for all integers $n\geq 1$, the function $\hat B_n$ is $L/n$-Lipschitz on $C^n$ with respect to $\dist_1^{(n)}$, where
$$L=\begin{cases}
    1 &\mbox{ if } \kappa\leq 0, \mbox{ independently of } C \\
    \pi/\alpha(\varepsilon,\kappa) &\mbox{ if } \kappa>0, \mbox{ with } C \mbox{ contained in a ball of radius } 1/2(D_\kappa/2-\varepsilon), \varepsilon>0.
    \end{cases}$$

\end{theorem}

\noindent Note that if $\kappa>0$, then by \eqref{eq:alpha}, $L\leq \frac{\pi^2}{4\varepsilon\sqrt\kappa}$.

\begin{proof}

When $\kappa\leq 0$, this result follows from \citep[Theorem 6.3]{sturm03} which, using Jensen's inequality, shows that the barycenter function is contractive on $\mathcal{P}^1(M)$ equipped with the Wasserstein distance $W_1$. More precisely, for any probability measure $\mu\in\mathcal P^1(M)$, we denote by $B(\mu)$ its (unique) barycenter. Then, for all $\mu,\nu\in\mathcal P^1(M)$, 
$$\dist(B(\mu),B(\nu))\leq W_1(\mu,\nu)$$ where $W_1(\mu,\nu)=\inf_{X\sim\mu,Y\sim\nu}\E[d(X,Y)]$. Now, fix two $n$-uples $(x_1,\ldots,x_n)$ and $(y_1,\ldots,y_n)$ in $M^n$ and set $\mu=n^{-1}\sum_{i=1}^n \delta_{x_i}$ and $\nu=n^{-1}\sum_{i=1}^n \delta_{y_i}$, so $B(\mu)=\hat B_n(x_1,\ldots,x_n)$ and $B(\nu)=\hat B_n(y_1,\ldots,y_n)$. Then, $W_1(\mu,\nu)\leq \frac{1}{n}(d(x_1,y_1)+\ldots+d(x_n,y_n))$, which can be seen by taking the coupling $(X,Y)$ of $\mu$ and $\nu$ such that $P(X=a_i, Y=b_i)=\frac{1}{n}, i=1,\ldots,n$.

When $\kappa>0$, a different technique is needed to obtain the desired result: For instance, a similar technique, based on Kendall's function \citep{kendall1991convexity}, appears in \citep[Theorem 8]{gietl2024lipschitz} but, because $W_1$ is replaced with $W_2$, this technique only yields a Lipschitz constant of order $1/\sqrt n$ instead of $n$. Let $C$ be a convex domain included in a ball $B$ of radius $1/2(D_\kappa/2-\varepsilon)$ for some $\varepsilon>0$ and let $x_1,\ldots,x_n,y_1,\ldots,y_n\in C$. 
Denote by $F(x)=\frac{1}{n}\sum_{k=1}^n\dist(x,x_k)^2$ and $G(x)=\frac{1}{n}\sum_{k=1}^n\dist(x,y_i)^2$, for all $x\in M$. Let $b_n=\hat B_n(x_1,\ldots,x_n)$ and $b_n'=\hat B_n(y_1,\ldots,y_n)$ be the minimizers of $F$ and $G$, respectively. By Proposition~\ref{lem:varianceineq}, $b_n,b_n'$ are both contained in $C$, so by Proposition~\ref{prop:uniquegeodesics}, there is a unique geodesic $\gamma\in\Gamma_{b_n,b_n'}$. Denote by $\tilde\gamma$ the (unique) geodesic from $b_n'$ to $b_n$, given by $\tilde\gamma(t)=\gamma(1-t)$ for $0\leq t\leq 1$. 
Moreover, by Lemma~\ref{lem:d2strconvex}, the map $F\circ\gamma$ is $\alpha(\varepsilon,\kappa)\dist(b_n,b_n')^2$-strongly convex, yielding that 
$$(F\circ\gamma)^-(1)\geq (F\circ\gamma)^+(0)+\alpha\dist(b_n,b_n')^2$$
where $f^+$ (resp. $f^-$) stands for the right (resp. left) derivative of a function $f:[0,1]\to\R$, and $\alpha=\alpha(\varepsilon,\kappa)$. Moreover, since $F\circ\gamma$ is minimized at $t=0$, $(F\circ\gamma)^+(0)\geq 0$, yielding 
\begin{equation} \label{eqnew:1}
    (F\circ\gamma)^-(1)\geq \alpha\dist(b_n,b_n')^2.
\end{equation}
Since $G\circ\tilde\gamma$ is minimized at $t=0$, we have $(G\circ\gamma)^-(1)=-(G\circ\tilde\gamma)^+(0)\leq 0$, and \eqref{eqnew:1} becomes
\begin{align}
    \alpha\dist(b_n,b_n')^2 & \leq (F\circ\gamma)^-(1) \nonumber \\
    & = (G\circ\gamma)^-(1)+\frac{1}{n}\sum_{k=1}^n\left((\dist_{x_k}^2-\dist_{y_k}^2)\circ\gamma\right)^-(1) \nonumber \\
    & \leq \frac{1}{n}\sum_{k=1}^n\left((\dist_{x_k}^2-\dist_{y_k}^2)\circ\gamma\right)^-(1) \nonumber \\
    & = \frac{1}{n}\sum_{k=1}^n\left((\dist_{y_k}^2-\dist_{x_k}^2)\circ\tilde\gamma\right)^+(0) \nonumber
\end{align}
where, for all $z\in M$, we denote by $\dist_z=\dist(z,\cdot)$. 
Now, by the first variation formula \citep[Corollary II.3.6]{bridson2013metric}, we obtain 
\begin{align*}
    \alpha & \dist(b_n,b_n')^2 \\
    & \leq \frac{1}{n}\sum_{k=1}^n \left(2\dist(x_k,b_n')\dist(b_n,b_n')\cos(\angle_{b_n'}(x_k,b_n))-2\dist(y_k,b_n')\dist(b_n,b_n')\cos(\angle_{b_n'}(y_k,b_n))\right),
\end{align*}
that is,
\begin{equation*}
    \alpha\dist(b_n,b_n')\leq \frac{2}{n}\sum_{k=1}^n \left(\dist(x_k,b_n')\cos(\angle_{b_n'}(x_k,b_n))-\dist(y_k,b_n')\cos(\angle_{b_n'}(y_k,b_n))\right).
\end{equation*}
Here, we have denoted by $\angle_z(x,y)$ the Alexandrov angle between the geodesics starting at $z$ and going to $x$ and $y$ respectively, for all $x,y,z\in C$, as per \citep[Definition I.1.2]{bridson2013metric} (see also \citep[Proposition III.3.1]{bridson2013metric}). In the sequel, let us assume that $b_n'\neq b_n$, as otherwise there is nothing to prove.

Fix $k\in\{1,\ldots,n\}$. Let $T_{b_n'}M$ be the completion of the tangent cone to $M$ at $b_n'$. By \citep[Theorem II.3.19]{bridson2013metric}, this is a $\CAT(0)$ space. In $T_{b_n'}M$, let $\tilde\eta:[0,\infty)\to T_{b_n'}M$ be the geodesic ray starting at the origin and such that $\tilde\eta(\dist(b_n,b_n'))=\Log_{b_n'}(b_n)$ (so $\tilde\eta$ has unit speed in $T_{b_n'}M$). Now, denoting by $\mathfrak b$ the Busemann function associated with $\tilde\eta$ (see Definition~\ref{def:Busemann}), it is easy to see that $\dist(x_k,b_n')\cos(\angle_{b_n'}(x_k,b_n))=-\mathfrak b(\Log_{b_n'}(x_k))$. Similarly, $\dist(y_k,b_n')\cos(\angle_{b_n'}(y_k,b_n))=-\mathfrak b(\Log_{b_n'}(y_k))$ and we obtain
\begin{equation*}
    \alpha\dist(b_n,b_n')\leq \frac{2}{n}\sum_{k=1}^n \left(\mathfrak b(\Log_{b_n'}(y_k))-\mathfrak b(\Log_{b_n'}(x_k))\right).
\end{equation*}
Since $\mathfrak b$ is $1$-Lipschitz (this can be seen from the definition of Busemann functions, using the reverse triangle inequality), we therefore obtain
\begin{equation*}
    \alpha\dist(b_n,b_n')\leq \frac{2}{n}\sum_{k=1}^n \dist_{T_{b_n'}M}(\Log_{b_n'}(y_k),\Log_{b_n'}(x_k))
\end{equation*}
where $\dist_{T_{b_n'}M}$ is the distance on $T_{b_n'}M$. Finally, after noting that $\max(\dist(b_n',x_k),\dist(b_n',y_k))\leq \frac{D_\kappa}{2}=\frac{\pi}{2\sqrt\kappa}$, Lemma~\ref{lem:comp_dist} yields the result (with $r=\frac{\pi}{2\sqrt\kappa}$).



\end{proof}

We also define another family of barycenter functions, which can be computed iteratively. Fix a positive integer $n$ and consider again a convex domain $C$ of a $\CAT$ space $(M,\dist)$. Let $t=(t_2,\ldots,t_n)\in (0,1)^{n-1}$. For all $x_1,\ldots,x_n\in C$, we define $\tilde B_n^{(t)}(x_1,\ldots,x_n)$ iteratively by setting $\tilde b_1=x_1$ and, for all $k=2,\ldots,n$, $\tilde b_k=\gamma_k(t_k)$ where $\gamma_k$ is the unique geodesic from $\tilde b_{k-1}$ to $x_k$, and setting $\tilde B_n^{(t)}(x_1,\ldots,x_n):=\tilde b_n$. This construction was introduced by \citep{sturm03} for $\CAT(0)$ spaces with $t_k=1/k,\, k=2,\ldots,n$ and later studied, for instance, by \citep{ohtapalfia} in general $\CAT(\kappa)$ spaces for any $\kappa\in\R$. When $(M,\dist)$ is a Euclidean space, the choice $t_k=1/k, k=2,\ldots,n$ yields $\tilde B_n^{(t)}=\hat B_n$, that is to say, $\bar x_k=(1-1/k)\bar x_{k-1}+(1/k)x_k$ where $\bar x_k$ is the average of $x_1,\ldots,x_k$. However, note that in general, and for any choice of the sequence $t=(t_1,\ldots,t_n)$, $\tilde B_n^{(t)}\neq \hat B_n$. Moreover, in general, $\tilde B_n^{(t)}$ is not symmetric in its arguments: This iterative construction depends on the order of the points $x_1,\ldots,x_n$. Finally, note that $\tilde B_n^{(t)}$ can be interpreted as the outcome of a proximal descent algorithm for the numerical computation of $\hat B_n$. Indeed, for $k=2,\ldots,n$, it holds that 
$$
\tilde b_k= \underset{x\in M}{\mathrm{argmin}} \left(\dist(x,x_k)^2+\frac{1}{2\lambda_k}\dist(x,\tilde b_{k-1})^2\right), \quad \mathrm{for}\,\, \lambda_k=\frac{t_k}{2(1-t_k)}.
$$
In other words, $\tilde b_k$ is given by the resolvent of the map $\dist(\cdot,x_k)^2$ evaluated at $\tilde b_{k-1}$, see \citep{ohtapalfia} for more details. Hence, if $X_1,\ldots,X_n$ are i.i.d random variables supported in $C$, then $\tilde B_n^{(t)}(X_1,\ldots,X_n)$ is the output of the stochastic proximal descent algorithm with varying step sizes $\lambda_k=t_k/(2(1-t_k)),\, k=2,\ldots,n$.
The following result gives a sensitivity analysis of $\tilde B_n^{(t)}$ for $t=(1/2,\ldots,1/n)$ in $\CAT(\kappa)$ spaces for $\kappa\leq 0$.

\begin{theorem} \label{thm:baryarelip}
    Let $(M,\dist)$ be a $\CAT(\kappa)$ space with $\kappa\leq 0$, $n\geq 1$ and $t=(1/2,\ldots,1/n)$. The function $\tilde B_n^{(t)}$ is $1/n$-Lipschitz. 
\end{theorem}

The proof of this theorem, available in \citep[Lemma 3.1]{Funano10}, is straightforward and proceeds by induction on $n$. However, we do not know how to show an analogous result in $\CAT(\kappa)$ spaces with $\kappa>0$.

\begin{openquestion}
    Given a $\CAT(\kappa)$ space $(M,\dist)$ with $\kappa>0$, a convex domain $C\subseteq M$ and a positive integer $n$, is there a non-trivial choice of step sizes $t\in (0,1)^{n-1}$ such that the iterated barycenter function $\tilde B_n^{(t)}$ is $L/n$-Lipschitz on $C^n$ for some $L>0$ that only depends on $\kappa$ and $C$?
\end{openquestion}

Of course, the choice of the step sizes should also be consistent with that of Theorem~\ref{thm:llnvariance} below, in order to keep our statistical guarantee in expectation, while also being able to prove a high probability bound, see Section~\ref{sec:Hpb}.

\section{The basics of the concentration of measure in metric spaces}\label{sec:Laplace}

The concentration of measure phenomenon was highlighted in the 1970's by V. Milman in the context of the asymptotics of Banach spaces. It was then very studied through its deep connections with a lot of mathematical objects, such as isoperimetry, Markov relaxation time, spectrum of  diffusion operators and large deviation theory to mention just a few. It is also understood in physics as the self-averaging property, i.e., the property for a random physical quantity to behave deterministically at a macroscopic level, when the number of particles tends to infinity. That is in agreement with the mathematical intuition that a metric measure space concentrates well if the Lipschitz functions over it are almost constant in the measure theoretic sense. Among other tools to handle the concentration phenomenon, such as concentration functions, expansion coefficients, or the observable diameter (see, e.g., \citep{Funano10}), we have chosen to underline the use of the Laplace transform in measure metric spaces.
In this section, $(M,\dist)$ is a metric space and $\mathcal F$ is the collection of $1$-Lipschitz functions $f:M\to\R$, that is, satisfying $|f(x)-f(y)|\leq \dist(x,y)$ for all $x,y\in M$.

\subsection{Laplace transform}

Let $X$ be a random variable in $M$ with at least one moment. It is clear that $f(X)$ also has one moment, for all $f\in\mathcal F$. Following \citep[Section 1.6]{Ledouxconcentration}, we define the Laplace transform of $X$ as 
\begin{equation}\label{eq:defLaplace}
\Lambda_X(\lambda) := \sup_{f\in \mathcal F} \E[e^{\lambda(f(X)-\E[f(X)])}], \quad \lambda\in\R.
\end{equation}

By symmetry of the class $\mathcal F$, i.e. ($f\in\mathcal F\iff -f\in\mathcal F$), $\Lambda_X$ is an even function and one can simply study it for $\lambda\geq 0$.
Before expanding on the use of this definition, let us review some properties that will be important in the sequel. Recall that for all integers $n\geq 2$, we equip the product space $M^n$ with the $\ell^1$-product distance defined as $\dist_1^{(n)}((x_1,\ldots,x_n),(y_1,\ldots,y_n))=\dist(x_1,y_1)+\ldots+\dist(x_n,y_n)$. 

\begin{lemma}[Tensorization property ] \label{produitdelaplace}
If $X_1,\ldots,X_n$ are independent random variables on $(M,\dist )$ with at least one moment, then the Laplace transform of the random vector $(X_1,\ldots,X_n)$ in the product space $(M^n,\dist_1^{(n)})$  satisfies $$\Lambda_{(X_1,\ldots,X_n)} \leq \Lambda_{X_1}\cdots \Lambda_{X_n}.$$
\end{lemma}

\begin{lemma}[Composition with Lipschitz functions] \label{lemma:LipLap}
    Let $(M^{(1)},d^{(1)})$ and $(M^{(2)},d^{(2)})$ be metric spaces and $\Phi:M^{(1)}\to M^{(2)}$ be a $L$-Lipschitz function, where $L>0$. Then, for all random variables $X$ in $M^{(1)}$ with at least one moment,
    $$\Lambda_{\Phi(X)}(\lambda)\leq \Lambda_X(\lambda L), \quad \forall \lambda\geq 0.$$
\end{lemma}

\begin{proof}
Let $f:M^{(2)}\to\R$ be a $1$-Lipschitz function. Then, for all $\lambda\geq 0$,
$$\E[e^{\lambda(f(\Phi(X))-\E[f(\Phi(X))])}] = \E[e^{\lambda L\frac{(f(\Phi(X))-\E[f(\Phi(X))])}{L}}] = \E[e^{\lambda L(g(X)-\E[g(X)])}]$$
where $g=(1/L)f\circ\Phi$ is a $1$-Lipschitz function. Hence, $\E[e^{\lambda(f(\Phi(X))-\E[f(\Phi(X))])}]\leq \Lambda_X(\lambda L)$ and one concludes by taking the supremum over all $1$-Lipschitz functions $f:M^{(2)}\to\R$.
\end{proof}

In the next two sections, we introduce two classes of random variables, based on an upper bound on their Laplace transform: Namely, sub-Gaussian and sub-Gamma random variables. In fact, we could introduce a whole family of such classes, e.g., Orlicz spaces. We restrict ourselves to these two families for simplicity, and because they are sufficient for our purposes which, here, are to extend Hoeffding and Berstein's inequalities to metric spaces.

\subsection{Sub-Gaussian random variables}\label{sec:subGauss}

In this section, we extend the notion of sub-Gaussian random variables, i.e., random variables in Euclidean spaces whose Laplace transform is bounded by that of a Gaussian variable, to metric spaces.

\begin{definition} \label{def:subGaussian}
	A random variable $X$ in $(M,\dist )$ is called $K^2$-sub-Gaussian ($K\geq 0$) if and only if $\Lambda_X(\lambda)\leq e^{\lambda^2K^2/2}$, for all $\lambda\in\R$.
\end{definition}

In other words, the random variable $X$ is $K^2$-sub-Gaussian if and only if $f(X)$ is $K^2$-sub-Gaussian for all $f\in\mathcal F$ (as per the standard definition for real random variables). 

\newpage
\begin{remark}\,
    \begin{itemize}
        \item Definition~\ref{def:subGaussian} is stronger than the standard definition of sub-Gaussian random variables in Euclidean spaces. Indeed, if $X$ is a random variable in $\R^p$ ($p\geq 1$), $X$ is said to be $K^2$-sub-Gaussian in the Euclidean, standard sense, if it satisfies Definition \ref{def:subGaussian} only with linear $1$-Lipschitz functions (see \citep[Section 2.5]{vershynin2018high}), that is,
        $$\E[e^{\lambda u^\top (X-\E X)}]\leq e^{\frac{\lambda^2 K^2}{2}}$$ 
        for all unit vectors $u\in\R^p$ and all $\lambda\in\R$.
        In order to see that Definition~\ref{def:subGaussian} is indeed stronger in Euclidean spaces, consider a random variable $X$ of the form $X=YZ$ where $Y$ has the standard Gaussian distribution in $\R^p$ ($p\geq 1$) and $Z$ be a Bernoulli random variable independent of $Y$ with $P(Z=0)=P(Z=1)=1/2$. Set $X=YZ$. One can easily verify that for all unit vectors $u\in\R^p$, $u^\top X$ is $1$-sub-Gaussian. However, there are $1$-Lipschitz functions $f:\R^p\to\R$ for which $f(X)$ is not $1$-sub-Gaussian. For instance, simply take $f=\|\cdot\|$ (Euclidean norm in $\R^p$). If $\|X\|$ was $K^2$-sub-Gaussian for some $K>0$, then it would necessarily hold that 
        $$P(\|X\|<\E[\|X\|]-\sqrt{p}/4)\leq e^{-p/(32K^2)}.$$
        However, since $\E[\|X\|]$ is approximately $\sqrt{p}/2$, when $p$ is large, it holds that the latter probability is at least $1/2$, which yields a contradiction if the dimension $p$ is much larger than $K^2$. 
        \item On the other hand, let us point out that if a random vector $X=(X_1,\ldots,X_p)$  in $\R^p$, with i.i.d coordinates, is $K^2$-sub-Gaussian in the usual sense ($K>0$), then it is $CpK^2$-sub-Gaussian in the sense of Definition \ref{def:subGaussian}, for some unversal constant $C>0$. Indeed, using \citep[Proposition 2.5.2 (d)]{vershynin2018high}, let us simply check that for all $1$-Lipschitz functions $f:M\to\R$, 
        $$\E\left[e^{\frac{1}{CpK^2}(f(X)-\E[f(X)])^2}\right]\leq 2$$
        if $C$ is chosen large enough, independently of $p$ and $K$. Let $Y$ be an independent copy of $X$. Then,
        \begin{align*}
            \E\left[e^{\frac{1}{CpK^2}(f(X)-\E[f(X)])^2}\right] & \leq \E\left[e^{\frac{1}{CpK^2}(f(X)-f(Y))^2}\right] \leq \E\left[e^{\frac{1}{CpK^2}\|X-Y\|^2}\right] \\
            & \leq \E\left[e^{\frac{4}{CpK^2}\|X-\E X\|^2}\right] = \left( \E\left[e^{\frac{4}{CpK^2}(X_1-\E X_1)^2}\right] \right)^{p} \\ 
            & \leq \left(e^{\frac{4}{CpK^2}K^2}\right)^p = e^{\frac{4}{C}} \leq 2
        \end{align*}
    for $C=4/\log(2)$. Here, the first inequality follows Jensen's inequality and the third one follows the triangle inequality. The fourth inequality is a consequence of \citep[Proposition 2.5.2 (c)]{vershynin2018high}, using the fact that $X_1$ is $K^2$-sub-Gaussian. 
    \item Definition~\ref{def:subGaussian} is perhaps the most canonical extension of the standard definition of sub-Gaussian random variables and it bears a deep connection with transportation inequalities. Indeed, Bobkov-Götze theorem (see \citep[Theorem 1.3]{bobkov1999exponential}) ensures that Definition \ref{def:subGaussian} is equivalent to the following transport-cost inequality $$ W_1\left(\mu,\nu \right) \leq K\sqrt{2\int_M \log\left(\frac{\diff\nu}{\diff\mu}\right)\diff\nu},$$ for all probability measures $\nu$ that are absolutely continuous with respect to $\mu$, and where $\mu$ is the probability distribution of $X$.
    
    \end{itemize}
\end{remark}

The following lemma is a straightforward generalization of the concentration properties of Euclidean sub-Gaussian distributions \citep[Proposition 2.5.2]{vershynin2018high}.

\begin{lemma}\label{lemmclassiquesousGauss}
Let $X$ be a random variable in $(M,\dist )$ and let $K>0$. The following statements are equivalent:
\begin{enumerate}
	\item $X$ is $K^2$-sub-Gaussian (in the sense of Definition~\ref{def:subGaussian})
	\item $f(X)$ is $K^2$-sub-Gaussian, for all $f\in\mathcal F$ (in the standard sense)
	\item $\sup_{f\in\mathcal F} P(f(X)-\E[f(X)]\geq t) \leq e^{-t^2/(2K^2)}$, for all $t\geq 0$.
\end{enumerate}
Moreover, the following implications hold: 
\begin{itemize}
    \item If $X$ is $K^2$-sub-Gaussian, then $\sup_{f\in\mathcal F}\E\left[e^{\frac{(f(X)-\E[f(X)])^2}{9K^2}}\right]\leq 2$.
    \item If $\sup_{f\in\mathcal F}\E\left[e^{\frac{(f(X)-\E[f(X)])^2}{2K^2}}\right]\leq 2$, then $X$ is $K^2$-sub-Gaussian.
\end{itemize}
\end{lemma}

For the sake of completeness, in the next two lemmas, we will describe the preservation of the sub-Gaussian property by tensorization and Lipschitz transformations.

\begin{proposition}[Tensorization]\label{produitsousgaussien}
Let $X_1,\ldots,X_n$ be independent random variables in $M$ such that each $X_i$ is $K_i^2$-sub-Gaussian for some $K_i>0$. Then, the $n$-uple $(X_1,\ldots,X_n)$ is $(K_1^2+\ldots+K_n^2)$-sub-Gaussian on the product metric space $(M^n, \dist_1^{(n)})$.
\end{proposition}

\begin{proof}
Let $X_i$ be $K_i^2$-sub-Gaussian, for each $i=1,\ldots,n$. Then, $\Lambda_{X_i}(\lambda)\leq e^{\lambda^2 K_i^2/2}$, for all $i=1,\ldots,n$ and $\lambda\geq 0$. Therefore, by using Lemma~\ref{produitdelaplace}, 
$$\Lambda_{(X_1,\ldots,X_n)}(\lambda) \leq \Lambda_{X_1}(\lambda)\ldots \Lambda_{X_n}(\lambda) \leq \prod_{i=1}^n e^{\lambda^2 K_i^2/2} = e^{\lambda^2(K_1^2+\ldots+K_n^2)/2},$$
for all $\lambda\geq 0$, which yields the result.
\end{proof}

\begin{proposition}[Composition with Lipschitz functions]\label{lipsousgauss}
Let $(M^{(1)},d^{(1)})$ and $(M^{(2)},d^{(2)})$ be metric spaces and $X$ be a random variables taking values in $M^{(1)}$. Let $K,L>0$. If $X$ is $K^2$-sub-Gaussian and $\Phi:M^{(1)}\to M^{(2)}$ is $L$-Lipschitz, then $\Phi(X)$ is $(L^2K^2)$-sub-Gaussian.
\end{proposition}

\begin{proof}
By using Lemma~\ref{lemma:LipLap}, for all $\lambda\geq 0$, $\Lambda_{\Phi(X)}(\lambda)\leq \Lambda_X(\lambda L) \leq e^{\lambda^2L^2K^2/2}$.
\end{proof}

\vspace{-2mm}

Let us conclude this section with two lemmas, which provide important examples of sub-Gaussian random variables. The first one is from \citep{Ledouxconcentration}; Similar to Hoeffding's lemma for real-valued random variables, it indicates that bounded random variables are always sub-Gaussian.

\begin{lemma}\citep[Proposition 1.16]{Ledouxconcentration} \label{lemma:Ledoux} \label{boundedsubgauss}
Let $X$ be a bounded random variable in the metric space $(M,\dist )$, i.e. $d(x_0,X)\leq C$ a.s. for some $x_0\in M$ and $C>0$. Then, $X$ is $4C^2$-sub-Gaussian. 
\end{lemma}

A second example of sub-Gaussian distribution can be constructed by designing a density with sufficient decay on a Riemannian manifold. 

\begin{lemma} \label{lemma:subG2}
    Let $M$ be a Riemannian manifold and $\dist$ be the corresponding Riemannian distance. Let $N$ be the dimension of $M$ and assume that $M$ has Ricci curvature bounded from below by some $R\in\R$. Let $X$ be a random variable in $M$ with a density $\phi$ with respect to the Riemannian volume such that 
    \begin{equation} \label{eq:density_bound}
        \phi(x)\leq Ce^{-\beta\dist(x,x_0)^2}, \quad\forall x\in M
    \end{equation}
    where $C,\beta>0$ and $x_0\in M$ are fixed. Then, $X$ is $K^2$ sub-Gaussian for some $K>0$ that depends on $C,\beta,R$ and $N$. 
\end{lemma}

A closed form for $K$ follows from the proof but we omit it here for simplicity. In fact, one does not need $M$ to be a Riemannian manifold in the previous lemma. Instead, assume that $(M,\dist)$ is a metric space that can be equipped with a reference measure $\mu$ such that the metric measure space $(M,\dist,\mu)$ satisfies the $(R,N)$-measure contraction property for some $R\in\R$ and $N>1$. This property generalizes the Ricci curvature lower bound and the dimension upper bound to abstract metric spaces. We refer the reader to \citep{ohta2007measure,sturmongeomI,sturmongeomII} for more details. In particular, any complete metric space with curvature bounded from below by $R$ in Alexandrov's sense (same definition as Definition~\ref{def:CAT}, but with reverse inequalities), equipped with its $N$-dimensional Hausdorff measure, satisfies the $((N-1)R,N)$-measure contraction property. An $N$-dimensional Riemannian manifold satisfies the $(K,N)$-measure contraction property if and only if its Ricci curvature is uniformly bounded from below by $R$. Now, the previous lemma can be extended to any metric measure space $(M,\dist,\mu)$ that satisfies the $(R,N)$-measure contraction property and any random variable $X$ in $M$ with density $\phi$ with respect to $\mu$, satisfying \eqref{eq:density_bound}.

\begin{lemma} 
	Let $(M,d,\mu)$ satisfying the $(R,N)$-measure contraction property for $K\in\R$ and $N>1$. Assume that $X$ is a random variable with value in $M$ and a density $\phi$ with respect to $\mu$, and such that 
	$$\phi(x)\leq Ce^{-\beta d(x,x_0)^2}, \forall x\in M,$$
for some given $C,\beta>0$ and $x_0\in M$. Then, $X$ is $K^2$-sub-Gaussian, for some $K>0$ that depends on $C, \beta$ and $K$.
\end{lemma}

Here again, a closed form for $K$ can be deduced from the proof, but we do not make it explicit here, for the sake of the simplicity of our presentation. The proof of these two lemmas follow from Bishop-Gromov volume comparison. Let us briefly sketch the argument, while deferring the complete proof to Appendix \ref{appendix:proofMCP}.
If $K>0$, then $M$ has finite diameter \citep{ohta2007measure}, bounded from above by $D=\pi\sqrt{\frac{N-1}{K}}$. Hence, $X$ is bounded and, by Lemma~\ref{lemma:Ledoux}, it is $K^2$-sub-Gaussian, with $K^2=4D^2$.
Now, assume that $K\leq 0$ and let $f\in\mathcal F$. By Jensen's inequality, for all $K>0$, $\DS \E[e^{\frac{(f(X)-\E[f(X)])^2}{2K^2}}] \leq \E[e^{\frac{(f(X)-f(Y))^2}{2K^2}}]$, where $Y$ is independent of $X$ and has the same distribution. Therefore, 
\begin{align}
	\E[e^{\frac{(f(X)-\E[f(X)])^2}{2K^2}}] & \leq \E[e^{\frac{d(X,Y)^2}{2K^2}}] \leq \E\left[e^{\frac{d(X,x_0)^2+d(Y,x_0)^2}{K^2}}\right] = \left(\E e^{\frac{d(X,x_0)^2}{K^2}}\right)^2 \nonumber \\
	& \leq \left(C\int_M e^{-\left(\beta-\frac{1}{K^2}\right)d(x,x_0)^2} \diff\mu(x)\right)^2. \nonumber 
\end{align}
Now the idea is to take $K$ large enough to get that the last integral is less than $2$. This is at this point that we need control on the growth of balls, and in particular the MCP$(K,N)$ condition gives the following generalized Bishop-Gromov volume comparison (see \citep{ohta2007measure}) $\forall r\geq 0$, $\mu(B(x_0,r)) \leq \int_0^r \mathbf{s}_K\left( \frac{t}{\sqrt{N-1}}\right)^{N-1}\diff t$, with $\mathbf{s}_0(t)=t$ and $\mathbf{s}_K(t)=\frac{1}{\sqrt{-K}}\sinh (\sqrt{-K}t)$, $K<0$. The proof ends with the integral being controlled on balls of large diameters, using Bishop-Gromov inequality.

\subsection{Sub-Gamma random variables}

\begin{definition}
    Let $\sigma^2>0$ and $c>0$. A random variable $X$ in $(M,\dist)$ is called $(\sigma^2,c)$-sub-Gamma if and only if its Laplace transform satisfies
    $$\Lambda_X(\lambda)\leq e^{\frac{\lambda^2\sigma^2}{2(1-\lambda c)}}, \quad \forall \lambda\in (0,c^{-1}).$$
\end{definition}

In other words, $X$ is $(\sigma^2,c)$-sub-Gamma if and only if $f(X)$ is a $(\sigma^2,c)$-sub-Gamma real random variable, as per \citep[Section 2.4]{boucheron2003concentration}. The following lemma shows that bounded random variables are sub-Gamma.

\begin{lemma}
    Let $X$ be a random variable in $(M,\dist)$. Assume that $\dist(X,x_0)\leq R$ almost surely, for some $x_0\in M$ and $R>0$. Then, $X$ has a second moment and, by denoting $\tilde\sigma^2=(1/2)\E[\dist(X,X')^2]$, where $X'$ is an independent copy of $X$, it holds that $X$ is $(\tilde \sigma^2,R)$-sub-Gamma.
\end{lemma}

Note that $\tilde\sigma^2\leq 2\sigma^2$ by the triangle inequality, where $\sigma^2=\inf_{x\in M}\E[\dist(X,x)^2]$ is the total variance of $X$. Hence, $X$ is also $(2\sigma^2,R)$-sub-Gamma. In fact, the inequality $\tilde\sigma^2\leq 2\sigma^2$ is tight up to universal constants. Indeed, by letting $F=\E[\dist(X,\cdot)^2]$, $\tilde\sigma^2=(1/2)\E[F(X')]\geq (1/2)\inf_{x\in M}F(x)=\sigma^2/2$.

\begin{proof}
    Let $f:M\to\R$ be a $1$-Lipschitz function and set $Y=f(X)$. Let us check that $Y$ is $(\sigma^2,R)$-sub-Gamma. First, one can verify that $\var(Y)=\sigma^2$. Moreover, $Y$ is bounded, since $|Y-f(x_0)|=|f(X)-f(x_0)|\leq \dist(X,x_0)\leq R$ almost surely. Therefore, $|Y-\E[Y]|\leq 2R$ almost surely, so, for all integers $p\geq 2$, $\E[|Y-\E[Y]|^p]=\E[|Y-\E[Y]|^2|Y-\E[Y]|^{p-2}]\leq \sigma^2(2R)^{p-2}$. Hence, for all $\lambda\in (0,R^{-1})$, we obtain
    \begin{equation*}
        \E[e^{\lambda(Y-\E[Y])}] \leq 1+\sigma^2\sum_{p\geq 2}\frac{\lambda^p(2R)^{p-2}}{p!}
        \leq 1+\frac{\sigma^2\lambda^2}{2}\sum_{p\geq 0}\lambda^pR^p = 1+\frac{\lambda^2\sigma^2}{2(1-\lambda R)}
        \leq e^{\frac{\lambda^2\sigma^2}{2(1-\lambda R)}}
    \end{equation*}
    where we used the facts that $2^{p-2}\leq p!$ for all $p\geq 2$ and $1+u\leq e^u$ for all $u\geq 0$. 
\end{proof}

Now, we show similar properties of sub-Gamma random variables as for sub-Gaussian ones. The first one is a tail bound that can be found in \citep[Section 2.4]{boucheron2003concentration}.

\begin{lemma} \label{lem:subGammatail}
    If $X$ is $(\sigma^2,c)$-sub-Gamma for some $\sigma^2,c>0$, then for all $f\in\mathcal F$ and $\delta\in (0,1)$, the following holds with probability at least $1-\delta$:
    $$f(X)\leq \E[f(X)]+\sigma\sqrt{2\log(1/\delta)}+c\log(1/\delta).$$
\end{lemma}

The following propositions concern tensorization and composition with Lipschitz functions.

\begin{proposition}[Tensorization]
    Let $X_1,\ldots,X_n$ be independent random variables such that each $X_i$ is $(\sigma_i^2,c_i)$-sub-Gamma for some $\sigma_i^2,c_i>0$. Then, the $n$-uple $(X_1,\ldots,X_n)$ is $(n\bar\sigma^2,c)$-sub-Gamma on the product metric space $(M^n,\dist_1^{(n)})$, with $\bar\sigma^2=n^{-1}(\sigma_1^2+\ldots+\sigma_n^2)$ is the average of the variances and $c=\max(c_1,\ldots,c_n)$.
\end{proposition}

\begin{proof}
    By Lemma~\ref{produitdelaplace},
    $\DS \Lambda_{X_1,\ldots,X_n}(\lambda) \leq \prod_{i=1}^n \Lambda_{X_i}(\lambda)
        \leq \prod_{i=1}^n e^{\frac{\lambda^2\sigma_i^2}{2(1-\lambda c_i)}}
        \leq e^{\frac{n\lambda^2\bar\sigma^2}{2(1-\lambda c)}}$,
    for all $\lambda\in (0,1/c)$.
\end{proof}

\begin{proposition}[Composition with Lipschitz functions]
    Let $(M^{(1)},d^{(1)})$ and $(M^{(2)},d^{(2)})$ be metric spaces and $X$ be a random variable taking values in $M^{(1)}$. Let $\sigma^2,c,L>0$. If $X$ is $(\sigma^2,c)$-sub-Gamma and $\Phi:M^{(1)}\to M^{(2)}$ is $L$-Lipschitz, then $\Phi(X)$ is $(L^2\sigma^2,Lc)$-sub-Gamma.
\end{proposition}

\begin{proof}
    By Lemma~\ref{lemma:LipLap},
    $\DS \Lambda_{\Phi(X)}(\lambda) \leq \Lambda_X(\lambda L) \leq e^{\frac{\lambda^2 L^2\sigma^2}{2(1-\lambda Lc)}}$, for all $\lambda\in (0,(Lc)^{-1})$.
\end{proof}

\section{Barycenter estimation} \label{sec:main}

Let $(M,\dist)$ be a $\CAT(\kappa)$ space with $\kappa\in\R$. Let $X_1,\ldots,X_n$ be i.i.d random variables supported in a convex domain $C\subseteq M$. In particular, if $\kappa>0$, then $X_1$ is bounded almost surely. If $\kappa\leq 0$, we further assume that $X_1$ has two moments. If $\kappa>0$, let $\varepsilon>0$ be such that $C$ is contained in a ball of radius $1/2(D_\kappa/2-\varepsilon)$. Then, by Proposition~\ref{lem:varianceineq}, $X_1$ has a unique barycenter, which lies in $C$ and which we denote by $b^*$. In the sequel, we call $b^*$ the \textit{population barycenter} of $X_1$. Our goal, here, is to estimate $b^*$ and derive the finite sample accuracy of our estimators, which we define below. An important quantity will be the total variance of $X_1$, which we denote by $\sigma^2$ and is defined as $\sigma^2=\E[\dist(X_1,b^*)^2]$. If $(M,\dist)$ is a Euclidean or Hilbert space, $\sigma^2$ is simply the trace of the covariance matrix of $X_1$.

\subsection{Empirical and iterated barycenters}

In the sequel, we denote by $\hat b_n=B_n(X_1,\ldots,X_n)$ the empirical barycenter, which again by Proposition~\ref{lem:varianceineq} is well defined and lies in $C$. Moreover, we denote by $\tilde b_n=\tilde B_n^{(t)}(X_1,\ldots,X_n)$ the iterated barycenter, where $t=(t_2,\ldots,t_n)\in (0,1)^{n-1}$ is a deterministic sequence to be specified later. We do not specify the dependence of $\tilde b_n$ on the choice of the sequence $t$ in our notation for the sake of simplicity. The estimator $\hat b_n$ will be referred to as the \textit{empirical barycenter} of $X_1,\ldots,X_n$ and $\tilde b_n$ as their \textit{iterated barycenter}. Our goal will be to derive upper bounds, both in expectation and with high probability, for the statistical error $\dist(b_n,b^*)$, where $b_n$ is either the empirical or the iterative barycenter, and $b^*$ is the population barycenter.

\subsection{Bounds in expectation}

First, we derive bounds for the expected error of $\hat b_n$. As in Lemma~\ref{lem:d2strconvex}, we let $\alpha(\varepsilon,\kappa) = (\pi-2\sqrt{\kappa}\varepsilon)\tan(\varepsilon\sqrt{\kappa})$ if $\kappa>0$ and $\varepsilon>0$.

\begin{theorem}\label{thm:llnempiricalbar}
    Let $(M,\dist)$ be a $\CAT(\kappa)$ space for some $\kappa\in \R$ and $C$ be a convex domain in $M$. If $\kappa>0$, let $\varepsilon>0$ be such that $C$ is enclosed in a ball of radius $1/2(D_\kappa/2-\varepsilon)$. Let $X_1,\ldots,X_n$ be i.i.d, square integrable random variables in $M$ such that $X_1\in C$ almost surely. Let $b^*$ be their population barycenter and $\hat b_n$ be their empirical barycenter. Then, 
    $$\E[\dist(\hat b_n,b^*)^2]\leq \frac{A\sigma^2}{n}$$
    where $A=2$ if $\kappa\leq 0$, $A=\frac{64\pi}{\alpha(\varepsilon,\kappa)^2}$ if $\kappa>0$. 
\end{theorem}

\begin{remark}\,
    \begin{itemize}
        \item When $\kappa\leq 0$, the same bound without the factor $2$ was obtained in \citep[Theorem 3]{fastconv}, assuming that the space also has curvature bounded from below in Alexandrov's sense, which implies that the tangent cone at a barycenter contains a Hilbert section, allowing to reduce the problem to the Hilbert case, after a few manipulations. If $(M,\dist)$ is a Hilbert space, the constant $2$ is indeed superfluous, and one actually has $\E[\dist(\hat b_n,b^*)^2]=\sigma^2/n$, which is an equality. In that case, only the last step of our proof is suboptimal, since $\E[\dist(X_1,X_1')^2]=2\sigma^2$ in Hilbert spaces. 
        \item A similar bound is also obtained in \citep[Theorem 1]{fastconv}, giving a bound of order $1/n$ for $\E[\dist(\hat b_n,b^*)^2]$, under a different set of assumptions. Precisely, they assume that the space has non-negative curvature and that geodesics are extendible and their bound depends on the level of extendibility.
        \item By \eqref{eq:alpha}, the constant $A$ is of order $1/\alpha^{2}$, where $\alpha=2$ if $\kappa\leq 0$ and $\alpha=\alpha(\varepsilon,\kappa)$ otherwise. The dependence on $\alpha$ of our upper bound in Theorem~\ref{thm:llnempiricalbar}  may be suboptimal in the small $\alpha$ regime, that is, when $\kappa>0$ and $\varepsilon\sqrt{\kappa}$ is small. If one allows $\varepsilon$ to depend on the sample size $n$ and $\varepsilon=\varepsilon_n\xrightarrow[n\to\infty]{}0$, we expect that slower rates may happen, as described by the smeariness phenomenon \citep{eltzner2019smeary,Eltzner2021}. In  such case, understanding the optimal dependence on $\alpha$ is essential. We leave this question for future work.
    \end{itemize}
\end{remark}

\begin{proof}

Our proof is inspired from \citep[Section 6.1]{escande2023concentration}.
Let $\alpha=\alpha(\varepsilon,\kappa)$ if $\kappa>0$ and $\alpha=2$ if $\kappa\leq 0$. Let $X_1',\ldots,X_n'$ be random variables in $M$ such that $X_1,\ldots,X_n,X_1',\ldots,X_n'$ are i.i.d. For $i=1,\ldots,n$, let $\hat b_n^{(i)}=B_n(X_1,\ldots,X_{i-1},X_i',X_{i+1},\ldots,X_n)$. 
Denote by $F(x)=\E[\dist(x,X_1)^2],\, x\in M$, the Fréchet function and by $F_n(x)=\frac{1}{n}\sum_{i=1}^n \dist(x,X_i)^2$, for all $x\in M$. The variance inequality of Proposition~\ref{lem:varianceineq} yields both that 
$$F(\hat b_n)\geq F(b^*)+\frac{\alpha}{2}\dist(\hat b_n,b^*)^2$$
as well as 
$$F_n(b^*)\geq F_n(\hat b_n)+\frac{\alpha}{2}\dist(\hat b_n,b^*)^2.$$
Taking expectations and summing both inequalities above, we obtain that 
$$\alpha\E[\dist(\hat b_n,b^*)^2]\leq \E[F(\hat b_n)-F_n(\hat b_n)].$$
Now, exchangeability of $X_1,\ldots,X_n,X_1',\ldots,X_n'$ yields that 
$$\E[F(\hat b_n)]=\frac{1}{n}\sum_{i=1}^n \E[\dist(X_i,\hat b_n^{(i)})^2].$$ 
Hence, we obtain that
\begin{equation} \label{eqn:1231}
    \alpha\E[\dist(\hat b_n,b^*)^2] \leq \frac{1}{n}\sum_{i=1}^n \E[\dist(X_i,\hat b_n^{(i)})^2-\dist(X_i,\hat b_n)^2].
\end{equation}

Now, let us distinguish two cases. 
\vspace{4mm}

\paragraph{\underline{Case 1: $\kappa\leq 0$.}}

If $(M,\dist)$ is a $\CAT(0)$ space, it satisfies the following quadruple inequality \citep[Proposition 2.4]{sturm03}:
$$\left(\dist(x,y)^2-\dist(x,y')^2\right)-\left(\dist(x',y)^2-\dist(x',y')^2\right)\leq 2\dist(x,x')\dist(y,y'),\, \forall x,x',y,y'\in M.$$ 
Fix $i\in\{1,\ldots,n\}$. Applying this inequality to the points $x=X_i,x'=X_i',y=\hat b_n,y'=\hat b_n^{(i)}$ yields that 
\begin{align*}
    \E[\dist(X_i,\hat b_n^{(i)})^2-\dist(X_i,\hat b_n)^2] & \leq \E[\dist(X_i',\hat b_n^{(i)})^2-\dist(X_i',\hat b_n)^2+2\dist(X_i,X_i')\dist(\hat b_n,\hat b_n^{(i)})] \nonumber \\
    & = \E[\dist(X_i,\hat b_n)^2-\dist(X_i,\hat b_n^{(i)})^2+2\dist(X_i,X_i')\dist(\hat b_n,\hat b_n^{(i)})]
\end{align*}
where we used again, in the second equality, exchangeability of $X_1,\ldots,X_n,X_1',\ldots,X_n'$ which implies that the pairs $(X_i,\hat b_n^{(i)})$ and $(X_i',\hat b_n)$ are identically distributed, as well as the pairs $(X_i,\hat b_n)$ and $(X_i',\hat b_n^{(i)})$. Therefore, \eqref{eqn:1231} yields that (recall that $\alpha=2$ here)
\begin{align*}
     2\E[\dist(\hat b_n,b^*)^2] & \leq \frac{1}{n}\sum_{i=1}^n \E[\dist(X_i,X_i')\dist(\hat b_n,\hat b_n^{(i)})] \\
     & \leq \frac{1}{n^2}\sum_{i=1}^n \E[\dist(X_i,X_i')^2] \\
     & = \frac{\E[\dist(X_1,X_1')^2]}{n} \\
     & \leq \frac{4\sigma^2}{n}.
\end{align*}
The second inequality used Theorem~\ref{thm:Lip_Bary} that states that $B_n$ is $1/n$-Lipschitz and the last inequality follows from the fact that $\E[\dist(X_1,X_1')^2]\leq \E[2(\dist(X_1,b^*)^2+\dist(X_1',b^*)^2)]=4\sigma^2$.
\vspace{4mm}

\paragraph{\underline{Case 2: $\kappa>0$.}}

Re-departing from \eqref{eqn:1231}, we have
\begin{align}
    \alpha \E[\dist(\hat b_n,b^*)^2]& \leq \frac{1}{n}\sum_{i=1}^n \E[\dist(X_i,\hat b_n^{(i)})^2-\dist(X_i,\hat b_n)^2] \nonumber \\
    & = \frac{1}{n}\sum_{i=1}^n \E\left[\left(\dist(X_i,\hat b_n^{(i)})-\dist(X_i,\hat b_n)\right)\left(\dist(X_i,\hat b_n^{(i)})+\dist(X_i,\hat b_n)\right)\right] \nonumber \\
    & \leq \frac{1}{n}\sum_{i=1}^n \E\left[\dist(\hat b_n,\hat b_n^{(i)})\left(\dist(X_i,\hat b_n^{(i)})+\dist(X_i,\hat b_n)\right)\right] \nonumber \\
    & \leq \frac{\pi}{n^2\alpha(\varepsilon,\kappa)}\sum_{i=1}^n \E\left[d(X_i,X_i')\left(\dist(X_i,\hat b_n)+\dist(X_i,\hat b_n^{(i)})\right)\right] \label{eqn:llnEB1}
\end{align}
where the second inequality is simply the reverse triangle inequality and the last one is a direct consequence of Theorem~\ref{thm:Lip_Bary}. Now, fix $i\in\{1,\ldots,n\}$. Since $\dist(X_i,\cdot)$ is continuous and convex on $B(x_0,r)$ (see \citep[Theorem 2.1]{afsari2011riemannian}), Jensen's inequality \citep[Theorem 25]{Yokota16} yields that 
$$\dist(X_i,\hat b_n)\leq \frac{1}{n}\sum_{j=1}^n \dist(X_i,X_j)$$
and
$$\dist(X_i,\hat b_n^{(i)})\leq \frac{1}{n}\left(\sum_{j\neq i} \dist(X_i,X_j')+\dist(X_i,X_i')\right).$$
Therefore, \eqref{eqn:llnEB1} implies that 
\begin{align}
    \frac{\alpha(\varepsilon,\kappa)}{2}\E[\dist(\hat b_n,b^*)^2] & \leq \frac{\pi}{n^3\alpha(\varepsilon,\kappa)}\sum_{i=1}^n \left(\sum_{j\neq i} 2\E[\dist(X_i,X_i')\dist(X_i,X_j')] +\E[\dist(X_i,X_i')^2] \right) \nonumber \\
    & \leq \frac{2\pi}{n^3\alpha(\varepsilon,\kappa)}\sum_{i=1}^n \sum_{j=1}^n \E[\dist(X_i,X_i')\dist(X_i,X_j')] \label{eqn:12312} \\
    & \leq \frac{8\pi}{n^2\alpha(\varepsilon,\kappa)}\sum_{i=1}^n\E[\dist(X_i,X_i')^2] \nonumber 
\end{align}
where we have used Cauchy-Schwarz inequality in the last line. Finally, using again the fact that $\E[\dist(X_i,X_i')^2]\leq 2\E[\dist(X_i,b^*)^2+\dist(X_i',b^*)^2] = 4\sigma^2$ (by the triangle inequality) concludes the proof of the theorem.
\end{proof}

In fact, minor modifications of our proofs allow to cover the heteroskedastic case, when $X_1,\ldots,X_n$ are independent but do not have the same distribution. However, we require that they share the same population barycenter. For instance, one can think of independent data with same population barycenter but different scales. 

\begin{theorem}[Error bound, heteroskedastic case] \label{thm:hetero1}
Let $X_1,\ldots,X_n$ be independent random variables with two moments and supported in a convex domain $C$ of $M$. If $\kappa>0$, let $\varepsilon>0$ be such that $C$ is enclosed in a ball of radius $(1/2)(D_\kappa/2-\varepsilon)$. Assume that all $X_i$'s share the same population barycenter $b^*$ and denote by $\sigma_i^2=\E[\dist(X_i,b^*)^2]$ the total variance of $X_i$, for $i=1,\ldots,n$. Then, by letting $\hat b_n=\hat B_n(X_1,\ldots,X_n)$, one has
$$\E[\dist(\hat b_n,b^*)^2]\leq \frac{\tilde A\bar\sigma_n^2}{n}$$
where $\bar\sigma_n^2=n^{-1}\sum_{i=1}^n \sigma_i^2$ and $\tilde A=2$ if $\kappa\leq 0$ and $\tilde A=\frac{32\sqrt{2}\pi}{\alpha(\varepsilon,\kappa)^2}$ if $\kappa>0$.
\end{theorem}

The proof of this theorem is deferred to Appendix \ref{sec:thmhetero1}, but let us note that the mapping $x\in M\mapsto \frac{1}{n}\sum_{i=1}^n\E[\dist(x,X_i)^2]$, plays the role of the population Fréchet function in the heteroskedastic setup and it is easy to see that it is strongly convex in $C$ and has a unique minimum given by $b^*$.

We now turn to iterated barycenters. 
First, one of the seminal results in the literature was proven by \citep{sturm03} for $\CAT(0)$ spaces. Namely, the proof of \citep[Theorem 4.7]{sturm03} gives the following bound, where the step sizes are set as $t_k=1/k,\, k=2,\ldots,n$,
\begin{equation} \label{eqn:Sturm}
    \E[\dist(\tilde b_n,b^*)^2]\leq \frac{\sigma^2}{n}.
\end{equation}
Recall that $\sigma^2=\E[\dist(X_1,b^*)^2]$ is the total variance of $X_1$. This gives the same bound as in Theorem~\ref{thm:llnempiricalbar} without the superfluous factor of $2$. Our next result provides an extension of this result in any $\CAT(\kappa)$ space, provided that the support of the data distribution is contained in a convex domain.

\begin{theorem}\label{thm:llnvariance}
     Assume that $\kappa>0$ and choose $t_k=\frac{2}{\alpha(\varepsilon,\kappa)k+2},\, k=2,\ldots,n$ in the definition of $\tilde b_n$. Then, 
     $$\E[\dist(\tilde b_n,b^*)^2]\leq \frac{32\sigma^2}{\alpha(\varepsilon,\kappa)^2(n+1)}.$$
\end{theorem}

\begin{remark}\,
    \begin{itemize}
        \item When $\kappa>0$, the step sizes $t_k$, or learning rates $\lambda_k=t_k/(2(1-t_k))$, are strictly larger than in the case $\kappa\leq 0$ when $k$ becomes sufficiently large. In other words, iterated barycenters (which, we recall, are also interpreted as the iterations of a stochastic proximal algorithm) learn the population barycenter more slowly when $\kappa>0$, which is consistent with an upper bound in Theorem~\ref{thm:llnvariance} that is larger than when $\kappa\leq 0$.
        \item The dependence on the strong convexity constant $\alpha(\varepsilon,\kappa)$ of the upper bound in Theorem~\ref{thm:llnvariance} is the same that we obtain in Theorem~\ref{thm:llnempiricalbar} for empirical barycenters. Again, we leave the question of optimality open for future work.
        \item When $M$ is a Riemannian manifold, one can easily check that iterated barycenters can be also seen as the outputs of a stochastic gradient descent algorithm. In such framework, asymptotic convergence guarantees were already obtained for stochastic gradient descent with general stepsizes in \citep{bonnabel2013stochastic} and results similar to ours are available for more general cost functions \citep{zhang2016first}. However, note that in \citep{zhang2016first}, all bounds depend on the size of the domain and not $\sigma^2$. Yet, if the domain is included in a ball of radius $R$ (with $R<D_\kappa$ if $\kappa>0$), then $\sigma^2\leq 4 R^2$ (as can be seen from the definition of $\sigma^2$), and one may have $\sigma^2\ll R^2$ in general. In the specific case of the Fréchet function, bounds similar to ours were also given in the Bures-Wasserstein manifold \citep{chewi2020gradient,altschuler2021averaging}, under a Polyak-Łojasiewicz condition, which is strictly weaker than strong convexity of the Fréchet function.

    \end{itemize}    
\end{remark}

\begin{openquestion}
    Find the optimal constants, in terms of $\varepsilon$ and $\kappa$, in the bounds given by Theorems~\ref{thm:llnempiricalbar},~\ref{thm:hetero1} and~\ref{thm:llnvariance}. In particular, do the bounds need to diverge as $\varepsilon\to 0$ (for fixed $\kappa>0$)?
\end{openquestion}

A key ingredient in the proof of Theorem~\ref{thm:llnvariance} is the following lemma.

\begin{lemma}\citep[Lemma 4.6]{ohtapalfia}\label{lem:ctrlresolvent}
    Assume that $M$ is CAT($\kappa$) for some $\kappa>0$. Let $B=B(x_0,r)$ for some $x_0\in M$ and $r<\pi/(4\sqrt\kappa)$. Let $f:M\to\R$ be a lower semi-continuous function that is convex on $B$ and $\lambda>0$. Fix $x\in B$ and let $z\in B$ minimizing $f(z)+\frac{1}{2\lambda}\dist(z,x)^2$. Then, for all $y\in B$, 
    \begin{equation}
        \dist(y,z)^2 \leq \dist(y,x)^2 -2\lambda\left[f(z)-f(y)\right].
    \end{equation}
\end{lemma}

Note that since $B$ is complete and $g:z\in B\mapsto f(z)+\frac{1}{2\lambda}\dist(z,x)^2$ is strongly convex (as the sum of a convex function $f$ and a strongly convex function $(2\lambda)^{-1}\dist(\cdot,x)^2$), it has one and one only minimizer in $B$.

\begin{proof}[Proof of Theorem~\ref{thm:llnvariance}]

Let $B$ be a ball of radius less than $\pi/(4\sqrt \kappa)$ such that $X_1\in B$ almost surely. Denote by $V_k=\E[\dist(\tilde b_k,b^*)^2]$ for $k=1,\ldots,n$. First, using induction on $k$, it is easy to see that 
\begin{equation} \label{eqn:Step12}
\E[\dist(\tilde b_k,X_{n+1})^2]\leq 4\sigma^2.
\end{equation}
Indeed, for $k=1$, this is follows from the series of inequalities 
\begin{align*}
    \E[\dist(X_1,X_{n+1})^2] & \leq \E\left[(\dist(X_1,b^*)+\dist(X_{n+1},b^*))^2\right] \\
    & \leq 2\E[\dist(X_1,b^*)^2]+2\E[\dist(X_{n+1},b^*)^2]=4\sigma^2,
\end{align*}
the first of which is the triangle inequality. Then, by convexity of $\dist(\cdot,X_{n+1})^2$ on $B$, for all $k=2,\ldots,n$, 
\begin{align*}
    \E[\dist(\tilde b_k,X_{n+1})^2] & \leq (1-t_k)\E[\dist(\tilde b_{k-1},X_{n+1})^2]+t_k\E[\dist(X_k,X_{n+1})^2] \\
    & \leq (1-t_k)\E[\dist(\tilde b_{k-1},X_{n+1})^2]+t_k(4\sigma^2)
\end{align*}
and the rest follows from the fact that $t_k\in [0,1]$.

Now, let us proceed to the proof of the theorem. Recall the step sizes $t_k$ and $\lambda_k$, related through the identity $t_k=\frac{2\lambda_k}{2\lambda_k+1}$, $k=2,\ldots,n$, in the definition of the iterated barycenters $\tilde b_1,\ldots,\tilde b_n$. Using Lemma~\ref{lem:ctrlresolvent}, we first write that
\begin{align}
    V_k & \leq V_{k-1}-2\lambda_k\left(\E[\dist(\tilde b_k,X_k)^2]-\E[\dist(X_k,b^*)^2]\right) \nonumber \\
    & = V_{k-1} -2\lambda_k\left(\E[\dist(X_k,\tilde b_{k-1})^2]-\E[\dist(X_k,b^*)^2]\right) +2\lambda_k\E[\dist(X_k,\tilde b_{k-1})^2-\dist(X_k,\tilde b_k)^2] \nonumber \\
    & = V_{k-1} -2\lambda_k\left(\E[\dist(X_k,\tilde b_{k-1})^2]-\E[\dist(X_k,b^*)^2]\right) +2\lambda_k t_k(2-t_k)\E[\dist(\tilde b_{k-1},X_k)^2] \label{eqn:Step121}
\end{align}
for $k=2,\ldots,n$.
First, note that since $\tilde b_{k-1}$ and $X_k$ are independent, the last expectation on the right hand side is equal to $\E[\dist(\tilde b_{k-1},X_k)^2]=\E[\dist(\tilde b_{k-1},X_{n+1})^2] \leq 4\sigma^2$ by \eqref{eqn:Step12}. Second, again by using the independence of $X_k$ and $\tilde b_{k-1}$, one can write $\E[\dist(X_k,\tilde b_{k-1})^2]=\E[F(\tilde b_{k-1})]$ where $F(x)=\E[\dist(X_k,x)^2],\, x\in M$, is the Fréchet function. Now, since $F$ is $\alpha(\varepsilon,\kappa)$-strongly convex on $B$, it holds that 
$F(\tilde b_{k-1})\geq F(b^*)+\frac{\alpha(\varepsilon,\kappa)}{2}\dist(\tilde b_{k-1},b^*)^2$ almost surely, and hence, \eqref{eqn:Step121} becomes
\begin{equation} \label{eqn:Setp122}
    V_k\leq \left(1-\lambda_k\alpha(\varepsilon,\kappa)\right)V_{k-1} +8\lambda_k t_k(2-t_k)\sigma^2 \leq \left(1-\lambda_k\alpha(\varepsilon,\kappa)\right)V_{k-1}+32\lambda_k^2\sigma^2
\end{equation}
using the facts that $2-t_k\leq 2$ and that $t_k\leq 2\lambda_k$ (recall that $t_k=2\lambda_k/(2\lambda_k+1)$).
Now, the result follows easily by induction on $k$. 

\end{proof}

Note that an asymptotic, non-quantitative version of Theorem~\ref{thm:llnvariance} was proven in \citep{ohtapalfia}, without any explicit choice of the step sizes.
The dependence on $\alpha(\varepsilon,\kappa)$ of the expected error of the iterative barycenter is better than that of the empirical barycenter (see Theorem~\ref{thm:llnempiricalbar}). Again, we do not know whether this dependence is optimal, neither for empirical or iterated barycenters, not in a minimax sense for the estimation of $b^*$. Contrary to the case of empirical barycenters, the proof of Theorem~\ref{thm:llnvariance} relies on the exchangeability of $X_1,\ldots,X_n$, because of the step given in \eqref{eqn:Step12}. 

When $\kappa=0$, it can be easily seen that Sturm's proof of \citep[Theorem 4.7]{sturm03} can be adapted to the heteroskedastic case, so as to obtain the following theorem.

\begin{theorem}[Heteroskedastic case, $\kappa=0$]
    Let $X_1,\ldots,X_n$ be independent random variables in a $\CAT(0)$ space $(M,\dist)$. Assume that all $X_i$'s have two moments and share the same population barycenter $b^*$. Then, 
    $$\E[\dist(\tilde b_n,b^*)^2]\leq \frac{\bar \sigma_n^2}{n}$$
where $\bar\sigma_n^2=\frac{\sigma_1^2+\ldots+\sigma_n^2}{n}$ and  $\sigma_i^2=\E[\dist(X_i,b^*)^2],\, i=1,\ldots,n$.
\end{theorem}

However the case $\kappa>0$ remains unclear to the authors, and we leave it as an open problem.
\begin{openquestion}
    When $\kappa>0$, does a bound similar to that of Theorem~\ref{thm:llnvariance} still holds in the heteroskedastic case?
\end{openquestion}

\subsection{High probability bounds} \label{sec:Hpb}

In this section, we prove bounds on the accuracy of $\hat b_n$ and $\tilde b_n$ that hold with high probability. Again, we assume that $(M,\dist)$ is a $\CAT(\kappa)$ space for some $\kappa\in\R$. If $\kappa\leq 0$, all the random variables $X_1,\ldots,X_n$ that are considered in this section are assumed to have two moments. If $\kappa>0$, they are all assumed to be almost surely contained in one and the same convex domain $C\subseteq M$ and we let $\varepsilon>0$ be such that $C$ is contained in a ball of radius $1/2(D_\kappa/2-\varepsilon)$.

\begin{theorem} \label{thm:mainthm}
    Assume that $X_1, \ldots,X_n$ are independent, have the same barycenter $b^*$ and that each $X_i$ is $K_i^2$-sub-Gaussian, for some $K_i>0$. For $i=1,\ldots,n$, let $\sigma_i^2$ be the total variance of $X_i$. Denote by $\bar\sigma^2=n^{-1}\sum_{i=1}^n\sigma_i^2$ and $\bar K^2=n^{-1}\sum_{i=1}^nK_i^2$. Then, for all $\delta\in (0,1)$, it holds with probability at least $1-\delta$ that
    $$\dist(\hat b_n,b^*)\leq \frac{\sqrt {\tilde A}\bar\sigma}{\sqrt n}+L\bar K\sqrt{\frac{\log(1/\delta)}{n}}$$
    where $\tilde A$ and $L$ are as in Theorems~\ref{thm:hetero1} and \ref{thm:Lip_Bary} respectively.
\end{theorem}

In the homosckedastic case, $\tilde A$ can be replaced with $A$ from Theorem~\ref{thm:llnempiricalbar}. By letting $\alpha=2$ if $\kappa\leq 0$ and $\alpha=\alpha(\varepsilon,\kappa)$ otherwise, \eqref{eq:alpha} implies that the high probability bound of Theorem~\ref{thm:mainthm} is, up to a universal multiplicative constant:
$$\dist(\hat b_n,b^*) \lesssim \frac{1}{\alpha\sqrt n}\left(\bar\sigma+\bar K\sqrt{\log(1/\delta)}\right).$$
As we have already mentioned above, the dependence of this bound on $\alpha$ may be suboptimal when $\alpha$ is small (i.e., $\kappa>0$ and $\varepsilon\sqrt{\kappa}$ is small). We leave this as an open question.

\begin{proof}
    The proof follows from the fact that $\hat b_n$, and hence so is $\dist(\hat b_n,b^*)$, is a Lipschitz function of $X_1,\ldots,X_n$, together with Propositions~\ref{produitsousgaussien}, \ref{lipsousgauss} and Theorem~\ref{thm:llnempiricalbar}.
\end{proof}

\begin{openquestion}
    Find the optimal constants in Hoeffding's bound on $CAT$ spaces.
\end{openquestion}

As a consequence of Lemma~\ref{boundedsubgauss}, we obtain the following version of Hoeffding's inequality for empirical barycenters, where we use the same notation as above. 

\begin{corollary} \label{cor:Hoeffding}
    Assume that $X_1,\ldots,X_n$ are independent and have the same barycenter $b^*$. Assume further that there exists $R>0$ with $R\leq 1/2(D_\kappa/2-\varepsilon)$ if $\kappa>0$, such that each $X_i$ is almost surely contained in some ball of radius $R$. Then, for all $\delta\in (0,1)$, it holds with probability at least $1-\delta$ that
    $$\dist(\hat b_n,b^*)\leq \frac{\sqrt{\tilde A}\bar\sigma}{\sqrt n}+2LR\sqrt{\frac{\log(1/\delta)}{n}}$$
    where $\tilde A$ and $L$ are as in Theorems~\ref{thm:hetero1} and \ref{thm:Lip_Bary} respectively. 
\end{corollary}

Again, in the homosckedastic case, $\tilde A$ can be replaced with $A$ from Theorem~\ref{thm:llnempiricalbar}. Note that when $\kappa\leq 0$, Corollary~\ref{cor:Hoeffding} was obtained independently in \citep{escande2023concentration} using a different approach, that is, based on the quadruple inequality, which characterizes $\CAT(0)$ spaces, and therefore cannot be extended to the setting of $CAT(\kappa)$ spaces for $\kappa>0$.
The next result is a generalization of Bernstein's inequality, which improves Hoeffding's inequality when $\bar\sigma\ll R$. Again, we use the same notation as above.

\begin{theorem} \label{thm:Bernstein}
    With the same assumptions as in Corollary~\ref{cor:Hoeffding}, for all $\delta\in (0,1)$, it holds with probability at least $1-\delta$ that 
    $$\dist(\hat b_n,b^*)\leq \frac{\sqrt{\tilde A}\bar\sigma}{\sqrt n}+2L\bar\sigma\sqrt{\frac{\log(1/\delta)}{n}}+LR\frac{\log(1/\delta)}{n}$$
    where $\tilde A$ and $L$ are as in Theorems~\ref{thm:hetero1} and \ref{thm:Lip_Bary} respectively. 
\end{theorem}

Again, letting $\alpha=2$ is $\kappa\leq 0$ and $\alpha=\alpha(\varepsilon,\kappa)$ when $\kappa>0$, the bound given in Theorem~\ref{thm:Bernstein} can be rewritten, up to a universal multiplicative constant, as 
$$\dist(\hat b_n,b^*) \lesssim \frac{1}{\alpha}\left(\frac{\bar\sigma}{\sqrt n}+\frac{\bar \sigma\sqrt{\log(1/\delta)}}{\sqrt n}+\frac{R\log(1/\delta)}{n}\right).$$

\newpage
\begin{remark}\label{rem:var-diam}\, 
\begin{itemize}
        \item For $\kappa\leq 0$, our versions of Hoeffding's and Bernstein's inequalities yield similar tail bounds for empirical barycenters as in Euclidean or Hilbert spaces. For $\kappa>0$, if $\varepsilon\sqrt\kappa$ is of constant order (e.g., if $M$ is a Euclidean sphere, $C$ is included in a spherical cap whose height is $1/6$th of the total height of the sphere), these inequalities also yield similar tail bounds for empirical barycenters as in Euclidean or Hilbert spaces, up to universal constants. However, again, in the small $\alpha$ regime, optimality of these bounds is unclear.
        
        \item It always holds that $\bar\sigma\leq 2R$ (since each $\sigma_i^2$ satisfies $\sigma_i^2\leq 4R^2$, as a direct consequence of the definition of $\sigma_i^2$). \\
        \item Our bounds are dimension free, in the sense that they do not require any notion of dimension (e.g., Hausdorff dimension) to be finite, as long as the $X_i$'s have finite second moment. 
    \end{itemize}
\end{remark}
Now, when $\kappa\leq 0$, we obtain similar results for iterated barycenters $\tilde b_n$. However, proving similar tail bounds in the case when $\kappa>0$ remains open. 

\begin{openquestion}
    Do similar high probability bounds hold for iterated barycenters when $\kappa>0$?
\end{openquestion}

\begin{theorem} \label{thm:HoeffBernCAT0}
    Let $(M,\dist)$ be a $\CAT(0)$ space. Let $X_1,\ldots,X_n$ be independent random variables with two moments, and having the same barycenter $b^*$. Let $\tilde b_n=\tilde B_n^{(t)}(X_1,\ldots,X_n)$ with $t=(1/2,1/3,\ldots,1/n)$. Let $\bar\sigma^2=n^{-1}\sum_{i=1}^n \sigma_i^2$, where $\sigma_1^2,\ldots,\sigma_n^2$ are the total variances of $X_1,\ldots,X_n$ respectively. Let $\delta\in (0,1)$.
    \begin{enumerate}
        \item[(i)] Assume that each $X_i$ is $K_i^2$-sub-Gaussian for some $K_i>0$. Then, with probability at least $1-\delta$, 
        $$\dist(\tilde b_n,b^*)\leq \frac{\bar\sigma}{\sqrt n}+\bar K\sqrt{\frac{\log(1/\delta)}{n}}$$
        where $\bar K^2=n^{-1}\sum_{i=1}^n K_i^2$.
        \item[(ii)] Assume that each $X_i$ is almost surely contained in some ball of radius $R>0$. Then, with probability at least $1-\delta$, 
        $$\dist(\tilde b_n,b^*)\leq \frac{\bar\sigma}{\sqrt n}+2R\sqrt{\frac{\log(1/\delta)}{n}}.$$
        \item[(iii)] The previous inequality can in fact be improved into 
        $$\dist(\tilde b_n,b^*)\leq \frac{\bar\sigma}{\sqrt n}+2\bar\sigma\sqrt{\frac{\log(1/\delta)}{n}}+R\frac{\log(1/\delta)}{n}.$$
    \end{enumerate}
\end{theorem}

\subsection{Application 1: Fast stochastic approximation of barycenters in $\CAT(0)$ spaces} \label{sec:GM}

Corollary~\ref{cor:Hoeffding} allows to obtain an algorithmic PAC guarantee for the stochastic approximation of barycenters of finitely many points in NPC spaces. Let $x_1,\ldots,x_n$ be given (deterministic) points in $M$. Here, the goal is to approximate their barycenter $b_n=B_n(x_1,\ldots,x_n)$. Recall that $b_n$ is the solution of an optimization problem, which may be hard to solve numerically. Fix some positive integer $m$ and follow the following steps:
\begin{itemize}
    \item Sample $m$ integers $I_1,\ldots,I_m$ independently, uniformly at random between $1$ and $n$;
    \item Set $X_1=x_{I_1}, \ldots, X_m=x_{I_m}$;
    \item Compute $\tilde b_m=\tilde B_m^{(t)}(X_1,\ldots,X_m)$ with step sizes $t=(1/2,1/3,\ldots,1/m)$.
\end{itemize}

The random variables $X_1,\ldots,X_m$ obtained in the second step are i.i.d with distribution $\mu=n^{-1}\sum_{i=1}^n\delta_{x_i}$, whose population barycenter is given by $b_n$.
In general, if $m$ is not too large, computing $\tilde b_m$ is simpler than computing $b_n$ directly, as long as one has access to an oracle that gives geodesics between any two points of $M$. The following result provides a PAC guarantee for $\tilde b_m$, as a stochastic approximation of $b_n$.

\begin{corollary} \label{cor:algo}
    Let $\varepsilon>0$ and $\delta\in (0,1)$. Let $D$ be the diameter of the set $\{x_1,\ldots,x_n\}$. Then, if $m\geq \frac{4D^2}{\varepsilon^2}\max(1,\log(1/\delta))$, it holds that $d(\tilde b_m,b_n)\leq \varepsilon$ with probability at least $1-\delta$.
\end{corollary}

\begin{proof}
Let $\sigma^2$ be the variance of $X_1$, i.e., $\sigma^2=\E[d(X_1,b_n)^2]$. Then, $\sigma^2\leq D^2$ (see Remark~\ref{rem:var-diam}). Therefore, Theorem~\ref{thm:HoeffBernCAT0} yields that with probability at least $1-\delta$, $d(\tilde b_m,b_n)\leq \frac{D}{\sqrt m}(1+\sqrt{\log(1/\delta)})$, which implies the desired result.
\end{proof}

Perhaps surprisingly, the algorithm complexity given by Corollary~\ref{cor:algo} is dimension free and only depends on $n$ through the computation of $D$ if unknown beforehand, and the bootstrapping procedure, that is, the sampling of uniform indices in $\{1,\ldots,n\}$.
In fact, if $\sigma^2\ll D^2$ this complexity can actually be further improved by using Theorem~\ref{thm:Bernstein}.

\begin{corollary} \label{cor:algo2}
    Let $\varepsilon>0$ and $\delta\in (0,1)$. Let $\tilde\sigma^2=\frac{1}{2n^2}\sum_{1\leq i,j\leq n} d(x_i,x_j)^2$ and $D$ be the diameter of the set $\{x_1,\ldots,x_n\}$. Then, if 
\begin{equation*} 
    m\geq \frac{16}{3}\max\left(\frac{\tilde\sigma^2}{\varepsilon^2},\frac{D}{\varepsilon}\right)\max(1,\log(1/\delta)),
\end{equation*}
it holds that $d(\tilde b_m,b_n)\leq \varepsilon$ with probability at least $1-\delta$.
\end{corollary}

Note that $\tilde\sigma^2=(1/2)\E[\dist(X_1,X_2)^2]\geq \sigma^2$ (see Remark~\ref{rem:var-diam}), so the proof of this corollary follows from Theorem~\ref{thm:Bernstein}. In comparison with the above numerical guarantees, \citep[Theorem 3.4]{LimPalfia14} gives a deterministic guarantee for finding an $\varepsilon$-approximation of the barycenter of $x_1,\ldots,x_n$, after $\frac{n(D^2+\sigma^2)}{\varepsilon^2}$ steps: The complexity of their algorithm is $n$ times worse than ours, where $n$ is the number of input points.


Here are two examples where this guarantee is useful. First, that of metric trees, where the computation of iterated barycenters simply requires to identify the shortest paths between any two points, which can be done efficiently. Another important example, in matrix analysis, is that of computing matrix geometric means. Recall that the geometric mean of positive definite matrices $A_1,\ldots,A_n\in \mathcal S_p$ ($n, p\geq 1$) is their barycenter, associated with the metric $d(A,B)=\|\log(A^{-1/2}BA^{-1/2})\|_{\textsf{F}}$, which makes $\mathcal S_p$ an NPC space \citep[Proposition 5]{bhatia2006riemannian}. The geometric mean of two matrices $A,B\in\mathcal S_p$ is the matrix $A\# B=A^{1/2}(A^{-1/2}BA^{-1/2})^{1/2}A^{1/2}$ and more generally, the geodesic segment between $A$ and $B$ is given by $\gamma_{A,B}(s)=A^{1/2}(A^{-1/2}BA^{-1/2})^sA^{1/2}$, also denoted by $A\#_s B$, for all $s\in [0,1]$. Hence, computing the sequence of iterated barycenters of positive definite matrices boils down to computing expressions such as $A^{1/2}(A^{-1/2}BA^{-1/2})^sA^{1/2}$ for $s=1/2, 1/3, \ldots$ which can be done exactly with matrix products and eigendecompositions, whose complexities depend on the size $p$ of the matrices. In fact, there are faster ways to compute good approximations of $A\#_s B$, for $A,B\in\mathcal S_p$ and $s\in [0,1]$, e.g., by using integral representations and Gaussian quadrature: We refer, for instance, to \citep{bhatia2009positive,simon2019loewner} for more details. 

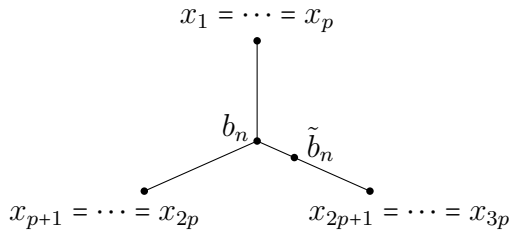
\begin{figure}[h]
    \centering
    \begin{center}
\begin{tikzpicture}
    \coordinate (A) at (0,2);   
    \coordinate (B) at (-1.5,0); 
    \coordinate (C) at (1.5,0);  
    
    \coordinate (Root) at (0, 2/3);

    \draw (Root) -- (A);
    \draw (Root) -- (B);
    \draw (Root) -- (C);

    \node at ([xshift=-8pt,yshift=5pt]Root) {\large $b_n$};
    \node[above] at (A) {\large $x_1 = \dots = x_p$}; 
    \node at ([xshift=-15pt,yshift=-10pt]B) {\large $x_{p+1} = \dots = x_{2p}$};
	
	\node at ([xshift=15pt,yshift=-10pt]C) {\large $x_{2p+1} = \dots = x_{3p}$};

    \node[circle, fill=black, inner sep=1pt] at (Root) {};
    \node[circle, fill=black, inner sep=1pt] at (A) {};
    \node[circle, fill=black, inner sep=1pt] at (B) {};
    \node[circle, fill=black, inner sep=1pt] at (C) {};

    \node[fill, circle, inner sep=1pt] at ($(Root) !0.33! (C)$) (P) {};

    \node at ([xshift=10pt,yshift=5pt]P) {\large $\tilde b_n$};

\end{tikzpicture}
\end{center}

    \caption{Barycenter on a metric tree ($n=3p$): Here, the iterated barycenter $\tilde b_n$ of $x_1,\ldots,x_n$ does not get any close to $b_n$ no matter how large $n$ is, if $x_1,\ldots,x_n$ are taken in this order.}
    \label{fig:example_tree}
\end{figure}

A natural question is whether the deterministic algorithm of \citep[Theorem 3.4]{LimPalfia14} could be improved. This algorithm consists of computing an iterated barycenter by making $K=\Omega(1/\varepsilon^2)$ passes through the whole set of points $x_1,\ldots,x_n$. That is, it computes an iterated barycenter
of $x_1,x_2,\ldots,x_n,x_1,x_2,\ldots,x_n,\ldots,x_1,x_2,\ldots,x_n$ ($K$ times) with appropriate step sizes. In fact, the example given in Figure~\ref{fig:example_tree} shows that one pass cannot be enough, in general. This stems from the fact that the order of the points $x_1,\ldots,x_n$ might not be favorable. However, we do not know, at this point, whether an initial random permutation could solve that issue. Note that the random algorithm that we proposed above, consists of randomly selecting points among $x_1,\ldots,x_n$ with replacement and we do not know whether this can be performed within $n$ steps without replacement to obtain a good approximation of $b_n$.

\begin{openquestion}
    How close, with high probability, is an iterated barycenter of a random permutation of $x_1,\ldots,x_n$, to their barycenter?
\end{openquestion}

We refer to \citep{shamir2016without} for related questions on sampling methods in stochastic optimization.

\subsection{Application 2: Parallelized barycenter estimation in symmetric spaces} \label{sec:distr}

In this section, we study the problem of parallelized computation of barycenters. The main feature of barycenters that break down in non-linear spaces is their associativity. For instance, given three points $x,y,z$, a barycenter of $x$ and of a barycenter of $y$ and $z$ is not, in general, a barycenter of $x,y$ and $z$. This is the main obstacle to the parallelization of the computation of a barycenter of a possibly large number of points. Here, we will focus on a case where the distribution of the data exhibits some form of symmetry. This will allow to design an estimator that can be computed in a distributed fashion while maintaining nearly the same statistical accuracy as the empirical barycenter. 

A natural framework to impose some symmetry is that of symmetric Riemannian manifolds. Let $(M,g)$ be a Riemannian manifold and $\dist$ be the distance induced by the Riemannian metric $g$. For a general introduction to Riemannian manifolds, including standard definitions and notation, which we employ here, we refer to \citep{lee2018introduction} or \citep{carmo1992Riemannian}. First, in order to fit the general framework of this work, let us assume that $M$ is simply connected and that its sectional curvature is uniformly bounded from above by some $\kappa\in\R$. By \citep[Theorem IX.5.1]{chavel2006riemannian}, this guarantees that $(M,\dist)$ is a $\CAT(\kappa)$ space. Let us also assume that $(M,g)$ is symmetric around $p$. That is, there exists an isometry $s_p$ (called symmetry around $p$) such that $s_p(p)=p$ and $\diff s_p(p)=-I_{T_pM}$, where $I_{T_pM}$ stands for the identity operator of the tangent space $T_pM$ of $M$ at $p$.

Now, let $X_1,\ldots,X_n$ be i.i.d random variables in $M$ and assume that:
\begin{itemize}
    \item If $\kappa>0$, $X_1\in B(p,1/2(D_\kappa/2-\varepsilon))$ almost surely, for some $\varepsilon>0$
    \item The distribution of $X_1$ is symmetric around $p$, that is, $s_p(X_1)$ and $X_1$ are identically distributed.
\end{itemize}

First, let us check that the barycenter $b^*$ of $X_1$ coincides with $p$.

\begin{lemma} \label{lem:barycentercenter}
    Under the above assumptions, $p$ is the unique barycenter of $X_1$.
\end{lemma}

\begin{proof}
    By Lemma~\ref{lem:varianceineq}, $X_1$ has a unique barycenter $b^*$ and $b^*\in B(p,1/2(D_\kappa/2-\varepsilon))$. Moreover, since $s_p$ is an isometry,
\begin{align*}
    \E[\dist(X_1,b^*)^2] & = \E[\dist(s_p(X_1),s_p(b^*))^2] \\
    & = \E[\dist(X_1,s_p(b^*))^2]
\end{align*}
so $s_p(b^*)$ must be equal to $b^*$. 

Assume, for the sake of contradiction, that $b^*\neq p$ and let $\gamma_1\in\Gamma_{p,b^*}$. For $t\in [0,1]$, let $\gamma_2(t)=s_p(\gamma_1(t))$. Since $s_p$ is an isometry, it is clear that $\gamma_2\in\Gamma_{p,s_p(b^*)}$. Then, by differentiating at $t=0$, we obtain that $\dot{\gamma}_2(0)=\diff s_p(p)(\dot\gamma_1(0))=-\dot\gamma_1(0)$. Therefore, by setting $\gamma(t)=\gamma_1(1-2t)$ for $0\leq t< 1/2$ and $\gamma(t)=\gamma_2(2t-1)$ for $1/2\leq t\leq 1$, $\gamma$ is a geodesic from $b^*$ to $s_p(b^*)$. However, since, by construction, $\gamma$ only takes values in a convex domain, this yields a contradiction, together with the fact that $b^*=s_p(b^*)$. 
\end{proof}

Now, let $P$ and $N$ be positive integers: $P$ will be the number of batches and $N$ the number of data within each batch. For the sake of simplicity, we assume that $n=PN$ and we let $I_1,\ldots,I_P$ be a partition of $\{1,\ldots,n\}$ into $P$ subsets of size $N$. For $j=1,\ldots,P$, let $Y_j$ be the empirical barycenter of the $X_i$'s, $i\in I_j$ and let $\hat b_n^{(P)}$ be the barycenter of $Y_1,\ldots,Y_P$. First, note that by Lemma~\ref{lem:varianceineq}, $Y_1,\ldots,Y_P$ as well as $\hat b_n^{(P)}$ are almost surely well defined, uniquely, and belong to $B(p,1/2(D_\kappa/2-\varepsilon))$ if $\kappa>0$. Then, we have the following result, where we keep the same notation as above.

\begin{theorem} \label{thm:parallel}
On top of the assumptions above, assume that $X_1$ is $K^2$-sub-Gaussian, for some $K>0$. For all $\delta\in (0,1)$, it holds with probability at least $1-\delta$ that
$$\dist(\hat b_n^{(P)},b^*) \leq \frac{\sqrt{A}\sigma}{\sqrt n}+L^2K\sqrt{\frac{\log(1/\delta)}{n}}$$
where $A$ and $L$ are as in Theorems~\ref{thm:llnempiricalbar} and \ref{thm:Lip_Bary} respectively.
\end{theorem}

While the rates are the same as for the empirical barycenter $\hat b_n$, only the constants are worse (compare with Theorem~\ref{thm:mainthm}, where the second term only had an $L$ instead of $L^2$). In the following proof, while we use Theorem~\ref{thm:mainthm}, we are in the heteroskedastic case, hence $\tilde A$ (from Theorem~\ref{thm:mainthm} is replaced with $A$ (from Theorem~\ref{thm:llnempiricalbar}), as explained right after the statement of Theorem~\ref{thm:mainthm}.

\begin{proof}
Let us check the following facts: (1) the population barycenter $Y_1$ coincides with $p$; (2) $Y_1$ is $L^2K^2/N$-sub-Gaussian, where $L$ is given in Theorem~\ref{thm:Lip_Bary} and (3) the total variance of $Y_1$ is bounded by $A\sigma^2/N$.
    
To check (1), by applying Lemma~\ref{lem:barycentercenter} to $Y_1$, it is enough to check that $s_p(Y_1)$ have the same distribution as $Y_1$.
First, note that $Y_1$ has the same distribution as $\hat B_N(X_1,\ldots,X_N)$ and, hence, as $\hat B_N(s_p(X_1),\ldots,s_p(X_N))$ by symmetry of the distribution of $X_1,\ldots,X_N$. Now, a similar argument as in the proof of Lemma~\ref{lem:barycentercenter} yields that $\hat B_N(s_p(X_1),\ldots,s_p(X_N))=s_p(B_N(X_1,\ldots,X_N))$ which has the same distribution as $s_p(Y_1)$. 

In order to check (2), recall that $X_1,\ldots,X_N$ are i.i.d $K^2$-sub-Gaussian and $B_N$ is $L/N$-Lipschitz (for $L$ given in Theorem~\ref{thm:Lip_Bary}), so $Y_1$ is $L^2K^2/N$-sub-Gaussian by Propositions~\ref{produitsousgaussien} and \ref{lipsousgauss}. 
    
Finally, to check (3), note that $Y_1$ has the same distribution as $\hat b_N=\hat B_N(X_1,\ldots,X_N)$, so Theorem~\ref{thm:llnempiricalbar} yields that its total variance is given by $\E[\dist(Y_1,p)^2]=\E[\dist(\hat b_N,p)^2]\leq \frac{A\sigma^2}{N}$.

Now, rewrite $\hat b_n^{(P)}$ as $\hat B_P(Y_1,\ldots,Y_P)$. Theorem~\ref{thm:mainthm} applied to the i.i.d random variables $Y_1,\ldots,Y_P$ yields that for all $\delta\in (0,1)$, it holds with probability at least $1-\delta$ that
$$\dist(\hat b_n^{(P)},p)\leq \frac{\sqrt{A}\sigma}{\sqrt n}+L^2K\sqrt{\frac{\log(1/\delta)}{n}}.$$

\end{proof}

When $\kappa\leq 0$, Theorem~\ref{thm:HoeffBernCAT0} allows to replace, in the definition of $\hat b_n^{(P)}$, empirical barycenters with inductive barycenters and obtain the same guarantee as in Theorem~\ref{thm:parallel} (with $A=2$ and $L=1$). That is, for all $j=1,\ldots,P$, $Y_j$ may be replaced with $Z_j:=\tilde B_N^{(t)}((X_i)_{i\in I_j})$ with $t=(1/2,\ldots,1/N)$ and $\hat b_n^{(P)}$ may be replaced with $\tilde b_n^{(P)}=\tilde B_P^{(s)}$ with $s=(1/2,\ldots,1/P)$. Indeed, the only thing to check is that the population of the $Z_j$'s is $p$. This can be easily done by induction on $N$, thanks to the following lemma.

\begin{lemma}
    Let $U_0$ and $U_1$ be two independent random variables in $M$ that are symmetric around $p$. Let $t\in [0,1]$ and set $U_t=\gamma(t)$ where $\gamma\in \Gamma_{U_0,U_1}$ (or, more simply, $U_t$ is the unique minimizer of $(1-t)\dist(U_0,x)^2+t\dist(U_1,x)^2, x\in M$). Then, the distribution of $U_t$ is symmetric around $p$. In particular, its population barycenter coincides with $p$. 
\end{lemma}

\begin{proof}
By definition of $U_t$, $s_p(U_t)$ is the unique minimizer $x\in M$ of $(1-t)\dist(U_0,s_p^{-1}(x))^2+t\dist(U_1,s_p^{-1}(x))^2=(1-t)\dist(s_p(U_0),x)^2+t\dist(s_p(U_1),x)^2$ since $s_p$ is an isometry. The conclusion follows from the fact that the pairs $(U_0,U_1)$ and $(s_p(U_0),s_p(U_1))$ are identically distributed.
\end{proof}

\begin{remark}
    In the absence of symmetry, there should be no hope to obtain such guarantees as in Theorem~\ref{thm:parallel} above, unless $N$ is of the same order as $n$ (i.e., $P$ is of constant order: The samples $X_1,\ldots,X_n$ are partitioned into very few batches). Indeed, in the absence of symmetry, the population barycenter of the $Y_j$'s, $j=1,\ldots,P$ (denote it by $b_N^*$), does not coincide with $b^*$ (the population barycenter of $X_1$) and might not even be close to it if $N$ is not large enough. Hence, $\hat b_n^{(P)}$ will have a large bias, i.e., it will concentrate around a point that is far from $b^*$. Perhaps the simplest and most convincing scenario is when $N=2$ and $P=n/2$. In that case, our results show that $\hat b_n^{(n/2)}$ will be $1/\sqrt n$-close to $b_2^*$, the population barycenter of the midpoint of $X_1$ and $X_2$, which is different from $b^*$ in general. 
    An open problem is whether, by taking $P$ of constant order, $\hat b_n^{(P)}$ concentrates significantly better around $b^*$ than $\hat b_N$ (i.e., the empirical barycenter of the first batch). We leave it as an open problem.
\end{remark}

\begin{openquestion}
    Given a fixed wall-clock time budget (here, $N$), is it statistically advantageous to compute several empirical barycenters in parallel and aggregate them, rather than run a single estimator?
\end{openquestion}

\section{The Riemannian case} \label{sec:Riemannian}

Here, we focus on the simpler case where $M$ is a smooth manifold and $\dist$ is the Riemannian distance induced by some Riemannian metric $g$ on $M$. The smooth structure allows us to simplify the analysis significantly while imposing a sub-Gaussian condition that is less stringent than Definition~\ref{def:subGaussian}, and that reduces to the standard sub-Gaussian definition when $M$ is Euclidean (see \eqref{eq:subG_Riem} below). This section is dedicated to deriving error bounds for empirical barycenters in that case. 

In what follows, for all $x\in M$ we denote by $T_xM$ the tangent space at $x$ and by $\langle u,v\rangle_x=g_x(u,v)$ the scalar product, inherited from the Riemannian metric $g$, of any two vectors $u,v\in T_xM$.
We assume that $M$ is simply connected and has sectional curvature uniformly bounded from above by $\kappa\in\R$. 
Then, by \citep[Thm IX.5.1]{chavel2006riemannian}, $M$ is a $\CAT(\kappa)$ space. For a general introduction to Riemannian manifolds, we refer to \citep{lee2018introduction} and \citep{carmo1992Riemannian}.

Let $X_1,\ldots,X_n$ be i.i.d random variables taking values in $M$. If $\kappa>0$, assume that $X_1\in B$ almost surely, where $B$ is a ball of radius $r=\frac{1}{2}(D_\kappa/2-\varepsilon)$ for some $\varepsilon>0$. Otherwise, set $B=M$ and simply assume that $X_1$ has a second moment. 

Finally, we assume that the injectivity radius of $M$ is greater than $r$, so the cut locus of any $x\in B$ does not intersect $B$. This last assumption ensures that for all $x\in B$, $\dist(\cdot,x)^2$ is smooth on $B$, with gradient given by $-2\Log_{\,\cdot}(x)$. 
Let $F(x)=\E[\dist(X_1,x)^2]$ and $F_n(x)=n^{-1}\sum_{i=1}^n \dist(X_i,x)^2, x\in M$, the population and empirical Fréchet functions. Then, both $F$ and $F_n$ are $\alpha$-strongly convex on $B$, where $\alpha=2$ if $\kappa\leq 0$ and $\alpha=\alpha(\varepsilon,\kappa)$ otherwise. As usual, let $b^*$ be the population barycenter of $X_1$ and $\hat b_n$ the empirical barycenter of $X_1,\ldots,X_n$. Then, $F$ and $F_n$ are both differentiable on $B$ and satisfy $\nabla F(b^*)=-2\E[\Log_{b^*}(X_1)]=0$ and $\nabla F_n(\hat b_n)=0$ by the dominated convergence theorem (where the first equality holds in $T_{b^*}M$ and the second one in $T_{\hat b_n}M$). Hence, $\Log_{b^*}X_1$ is a centered random vector in $T_{b^*}M$. Moreover, $\alpha$-strong convexity of $F_n$ yields that 
\begin{equation}
    \frac{\alpha}{2}\dist(\hat b_n,b^*)\leq \|\nabla F_n(b^*)\|_{b^*}=\|2n^{-1}\sum_{i=1}^n \Log_{b^*}X_i\|_{b^*}.
\end{equation}
This follows from a standard argument that can be readily adapted from the Euclidean to the Riemannian case. Let $f:M\to\R$ be differentiable and $\alpha$-strongly convex on $B$, with global minimizer $x^*\in B$. Let $x\in B$ and $\gamma$ be the unique geodesic from $x$ to $x^*$. Then, strong convexity implies that for all $t\in (0,1)$, 
\begin{align*}
    (1-t)f(x)+tf(x^*) & \geq f(\gamma(t))+\frac{\alpha}{2}t(1-t)\dist(x,x^*)^2 \\ 
    & \geq f(x)+t\langle \gamma'(0),\nabla f(x)\rangle_x + \frac{\alpha}{2}t(1-t)\dist(x,x^*)^2 \\
    & = f(x)+t\langle \Log_x(x^*),\nabla f(x)\rangle_x + \frac{\alpha}{2}t(1-t)\dist(x,x^*)^2 
\end{align*}
where the second inequality is a consequence of the convexity of $f\circ\gamma$. Moreover, since $f(x^*)\leq f(x)$, after dividing by $t$ and letting $t\to 0$, we obtain that
\begin{align*}
    \frac{\alpha}{2}\dist(x,x^*)^2 & \leq -\langle \Log_x(x^*),\nabla f(x)\rangle_x \\
    & \leq \|\Log_x(x^*)\|_x \|\nabla f(x)\|_x \\
    & = \dist(x,x^*)\|\nabla f(x)\|_x
\end{align*}
where we used the Cauchy-Schwarz inequality. Thus, $(\alpha/2)\dist(x,x^*)\leq \|\nabla f(x)\|_x$.
Now, assume that $\Log_{b^*}X_1$ is $K^2$-sub-Gaussian for some $K>0$, in the standard, Euclidean sense. That is, for all $u\in T_{b^*}M$, $\langle u,\Log_{b^*}X_1\rangle_{b^*}$ is $K^2\|u\|_{b^*}^2$-sub-Gaussian, \textit{i.e.}, 

\begin{equation} \label{eq:subG_Riem}
    \E[e^{\langle u,\Log_{b^*}X_1 \rangle_{b^*}}]\leq e^{K^2\|u\|_{b^*}^2/2}, \, \forall u\in T_{b^*}M.
\end{equation}

If $\kappa>0$, this is automatically satisfied with $K=2r$, since $\|\Log_{b^*} X_1\|_{b^*}=\dist(b^*,X_1)\leq 2r$ almost surely.
Note that $\Log_{b^*} X_1$ is a centered, square-integrable random vector in $T_{b^*}M$. Denoting by $\Sigma$ its covariance operator, we have that $\sigma^2=\E[\dist(X_1,b^*)^2]=\E[\|\Log_{b^*}X_1\|_{b^*}^2]=\tr(\Sigma)$. 

\begin{lemma}
    Let $Y_1,\ldots,Y_n$ ($n\geq 1$) be a centered, squared integrable random vectors in a Euclidean space $E$ with scalar product denoted by $\langle\cdot,\cdot\rangle$ and Euclidean norm $\|\cdot\|$. Let $\Sigma$ be the covariance operator of $Y_1$ and denote by $\bar Y_n=n^{-1}\sum_{i=1}^n Y_i$.
    \begin{itemize}
        \item If there is a positive number $K$ such that $\E[e^{\langle u,Y_i\rangle}]\leq e^{K^2\|u\|^2/2}$ for all $u\in E$, then for all $\delta\in (0,1)$, it holds with probability at least $1-\delta$ that
        $$\|\bar Y_n\| \leq CK\sqrt{\frac{d+\log(1/\delta)}{n}}$$
        for some universal constant $C>0$, where $d$ is the dimension of $E$.
        
        \item If there exists $R>0$ such that $\|Y_1\|\leq R$ almost surely, then for all $\delta\in (0,1)$, it holds with probability at least $1-\delta$ that
        $$\|\bar Y_n\| \leq \sqrt{\frac{\tr(\Sigma)}{n}} + 2R\sqrt{\frac{\log(1/\delta)}{n}}.$$
    \end{itemize}
\end{lemma}

The first part of the lemma follows from generic chaining arguments \citep{talagrand2014upper}, see \citep[Exercice 6.3.5]{vershynin2018high}, while the second part is a simple consequence of the bounded differences inequality \citep[Theorem 6.2]{boucheron2003concentration}. Note that if $\|Y_1\|\leq K$ almost surely, then it satisfies the condition of the first part of the lemma, but the concentration inequality is tighter ($K$ can be chosen as $K^2=R^2$ and $\tr(\Sigma)\leq dR^2$) and dimension free in that case.
Hence, we obtain the following high probability bounds for $\dist(\hat b_n,b^*)$. 

\begin{theorem}[Unbounded case, non-positive curvature] \label{thm:Riem_unbounded}
Let $M$ be a simply connected Riemannian manifold of dimension $d\geq 1$ with non-positive sectional curvature and with infinite injectivity radius. Let $X_1,\ldots,X_n$ be i.i.d random variables in $M$ with two moments. Let $b^*$ be their population barycenter, $\sigma^2=\E[\dist(X_1,b^*)^2]$ be their total variance, and assume that $\Log_{b^*}(X_1)$ is $K^2$-sub-Gaussian for some $K>0$ in the sense of \eqref{eq:subG_Riem}.
Then, for all $\delta\in (0,1)$, it holds with probability at least $1-\delta$ that 
$$\dist(\hat b_n,b^*)\leq CK\sqrt{\frac{d+\log(1/\delta)}{n}}$$
where $C>0$ is a universal constant.
\end{theorem}

\begin{theorem}[Bounded case]\label{thm:Riem_bounded}
    Let $M$ be a simply connected Riemannian manifold of dimension $d\geq 1$ with sectional curvature uniformly bounded by $\kappa\in\R$. Let $X_1,\ldots,X_n$ be i.i.d random variables in $M$ that are almost surely contained in some ball $B$ of radius $r>0$. If $\kappa>0$, assume that $r= 1/2(D_\kappa/2-\varepsilon)$ for some $\varepsilon>0$. Assume that the cut locus of any point of $B$ does not intersect $B$. Let $b^*$ the population barycenter of $X_1$ and $\sigma^2=\E[\dist(X_1,b^*)^2]$ be its total variance. Then, for all $\delta\in (0,1)$, it holds with probability at least $1-\delta$ that 
    $$\dist(\hat b_n,b^*)\leq \frac{\sigma}{\alpha\sqrt n}+\frac{2r}{\alpha}\sqrt{\frac{\log(1/\delta)}{n}}$$
    where $\alpha=2$ if $\kappa\leq 0$ and $\alpha=\alpha(\varepsilon,\kappa)$ otherwise.
\end{theorem}

The bounds given in these two theorems look worse than the ones we obtained in a more general framework. Indeed, both dependences on $\alpha(\varepsilon,\kappa)$ (in the small $\varepsilon$ regime) and in $(\sigma^2,K^2)$ are deteriorated. However, the assumption we made on the distribution of $X_1$ in Theorem~\ref{thm:Riem_unbounded} is less stringent than, say, in Theorem~\ref{thm:mainthm}, and it is more standard. Indeed, we do not require that all Lipschitz functions of $X_1$ are sub-Gaussian, but only those of the form $\langle u,\Log_{b^*}X_1\rangle_{b^*}$ for $u\in T_{b^*}M$ (see Section~\ref{sec:subGauss} above). Moreover, the bound obtained in Theorem~\ref{thm:Riem_bounded} has a better dependence on $\alpha$ than Corollary~\ref{cor:Hoeffding} in the small $\alpha$ regime, that is, when $\kappa>0$ and $\varepsilon$ is very small.

\section*{Acknowledgements} This work was funded by the French Agence Nationale de la Recherche, as part of the project
ANR-22-ERCS-0015-01.

\bibliographystyle{alpha}
\bibliography{biblio}

\appendix

\section{Remaining proofs} \label{sec:Appendix}

\subsection{Proof of Theorem~\ref{thm:hetero1}} \label{sec:thmhetero1}

In what follows, denote by $\alpha=2$ if $\kappa\leq 0$ and $\alpha=\alpha(\varepsilon,\kappa)$ otherwise, where $\varepsilon>0$ is such that $C$ is contained in a ball of radius $(1/2)(D_\kappa/2-\varepsilon)$. For all $x\in M$, denote by $F(x)=\frac{1}{n}\sum_{i=1}^n\E[\dist(x,X_i)^2]$ and $F_n(x)=\frac{1}{n}\sum_{i=1}^n\dist(x,X_i)^2$. First, note that $F$ is $\alpha$-strongly convex in $C$. Moreover, for all $i=1,\ldots,n$ and all $x\in M$ with $x\neq b^*$, $\E[\dist(x,X_i)^2]>\E[\dist(b^*,X_i)^2]$ since $b^*$ is the unique barycenter of $X_i$, yielding that $F(x)\geq \frac{1}{n}\sum_{i=1}^n \E[\dist(b^*,X_i)^2]=F(b^*)$. Hence, $b^*$ is the unique minimizer of $F$. Recall also that $b^*\in C$ and $\hat b_n\in C$ almost surely, by Proposition~\ref{lem:varianceineq}. Hence, we can write that $\frac{\alpha}{2}\dist(\hat b_n,b^*)^2\leq F(\hat b_n)-F(b^*)$. 

Let $X_1',\ldots,X_n'$ be random variables in $M$ such that $X_1,\ldots,X_n,X_1',\ldots,X_n'$ are independent and $X_i$ has the same distribution as $X_i'$, for all $i=1,\ldots,n$. Independence of $(X_1,\ldots,X_n)$ and $X_i'$, for $i=1,\ldots,n$, yields that $\E[F(\hat b_n)]=n^{-1}\sum_{i=1}^n\E[\dist(\hat b_n,X_i')^2]$.

For $i=1,\ldots,n$, let $\hat b_n^{(i)}=\hat B_n(X_1,\ldots,X_{i-1},X_i',X_{i+1},\ldots,X_n)$. Then, for each $i=1,\ldots,n$, $\dist(\hat b_n,X_i')$ and $\dist(\hat b_n^{(i)},X_i)$ have the same distribution, yielding that 
\begin{equation} \label{eqn:exchang}
    \E[F(\hat b_n)]=\frac{1}{n}\sum_{i=1}^n\E[\dist(\hat b_n^{(i)},X_i)^2]
\end{equation}

Now, \eqref{eqn:1231} from the proof of Theorem~\ref{thm:llnempiricalbar} still holds and in the case where $\kappa\leq 0$, the same whole argument works again. Let us only focus on the case $\kappa>0$, which requires slightly more care. In that case, \eqref{eqn:12312} did not require the $X_i$'s to have the same distribution, so we can write again 
\begin{align}
    \frac{\alpha(\varepsilon,\kappa)}{2}\E[\dist(\hat b_n,b^*)^2] 
    & \leq \frac{2\pi}{n^3\alpha(\varepsilon,\kappa)}\sum_{i=1}^n \sum_{j=1}^n\E[\dist(X_i,X_i')\dist(X_i,X_j')] \nonumber \\
    & \leq \frac{2\pi}{n^3\alpha(\varepsilon,\kappa)}\sum_{i=1}^n \sum_{j=1}^n \E[\dist(X_i,X_i')^2]^{1/2}\E[\dist(X_i,X_j')^2]^{1/2} \nonumber \\
    & \leq \frac{8\pi}{n^3\alpha(\varepsilon,\kappa)}\sum_{i=1}^n \sum_{j=1}^n \sigma_i(2\sigma_i^2+2\sigma_j^2)^{1/2} \nonumber \\
    & \leq \frac{8\sqrt 2\pi}{n^3\alpha(\varepsilon,\kappa)}\sum_{i=1}^n\sum_{j=1}^n(\sigma_i^2+\sigma_i\sigma_j) \nonumber \\
    & \leq \frac{8\sqrt 2\pi}{n^3\alpha(\varepsilon,\kappa)}\sum_{i=1}^n\sum_{j=1}^n(\sigma_i^2+(1/2)\sigma_i^2+(1/2)\sigma_j^2) \nonumber \\
    & = \frac{16\sqrt 2\pi}{n^3\alpha(\varepsilon,\kappa)}\sum_{i=1}^n \sigma_i^2 \nonumber \\
    & = \frac{16\sqrt 2 \pi \bar\sigma_n^2}{n\alpha(\varepsilon,\kappa)} \nonumber 
\end{align}
where we used the Cauchy-Schwarz inequality in the third line and that 
$$\E[\dist(X_i,X_j')^2]\leq 2\E[\dist(X_i,b^*)^2+\dist(X_j',b^*)^2]=2\sigma_i^2+2\sigma_j^2$$ 
for all $i,j\in\{1,\ldots,n\}$ in the fourth line.

\subsection{Proof of Lemma~\ref{lemma:subG2}}\label{appendix:proofMCP}

The proof of this lemma makes use of Bishop-Gromov comparison theorem. Assume that $(M,\dist,\mu)$ is a metric measure space that satisfies the $(\underline\kappa,N)$-measure contraction property. First, by \citep[Theorem 4.2]{ohta2007measure} (Bonnet-Myers' theorem in the case of a Riemannian manifold), if $\underline\kappa>0$, then $M$ has finite diameter, bounded from above by $\pi\sqrt{(N-1)/\underline\kappa}$. Hence, $X$ is bounded and, by Lemma~\ref{lemma:Ledoux}, it is $K^2$-sub-Gaussian, with $K^2=4\pi^2/\underline\kappa$. In the rest of the proof, assume that $\underline\kappa\leq 0$. Moreover, we assume that $N\geq 2$ is an integer, for simplicity. Then, by \citep[Theorem 5.1]{ohta2007measure} (Bishop-Gromov's theorem in the case of a Riemannian manifold), for all $x_0\in M$ and for all $r\geq 0$, it holds that 
$$\mu(B(x_0,r))\leq V_{N,\underline\kappa}(r),$$
where $V_{N,\kappa}(r)$ is the volume of any ball of radius $r$ in the $N$-dimensional hyperbolic space of constant curvature $\underline\kappa$ (which we identify with $\R^N$ is $\kappa=0$).

It is known \citep[Section III]{chavel2006riemannian} that 
$$V_{N,\kappa}(r)=c_{N-1}\int_0^r \left(\frac{\sinh(\sqrt{-\underline\kappa}t)}{\sqrt{-\underline\kappa}}\right)^{N-1}\diff t$$
where $c_{N-1}=\frac{2\pi^{N/2}}{\Gamma(N/2)}$ and where the integral should be understood as $r^N/N$ if $\underline\kappa=0$. If $\underline\kappa<0$, we readily obtain the inequality
$$V_{N,\underline\kappa}(r)=\frac{c_{N-1}e^{(N-1)\sqrt{-\underline\kappa}r}}{(N-1)(-\underline\kappa)^{N/2}}.$$

Now, let us show that for all $\alpha>0$, $I(\alpha):=\int_M e^{-\alpha d(x,x_0)^2}\diff\mu(x)$ is finite. This will be the key of the proof. 

For any choice of $c>0$,
\begin{align}
	I(\alpha) & = \sum_{r=0}^\infty \int_{B(x_0,c(r+1))\setminus B(x_0,cr)} e^{-\alpha d(x,x_0)^2}\diff\mu(x) \nonumber \\
	& \leq \sum_{r=0}^\infty e^{-\alpha c^2r^2}V_{N,\underline\kappa}(c(r+1)). \label{eqn:intsum}
\end{align}

For simplicity, let us distinguish the two cases when $\underline\kappa=0$ or $\underline\kappa<0$. First, assume $\underline\kappa=0$. Then, \eqref{eqn:intsum} with $c=1/\sqrt\alpha$ yields 
$$I(\alpha) \leq \frac{c_{N-1}}{(N-1)\alpha^{(N-1)/2}}\sum_{r=0}^\infty e^{-r^2}(r+1)^N<\infty.$$

Now, let us assume that $\underline\kappa<0$. Then, \eqref{eqn:intsum} with $c=1/\sqrt\alpha$ again yields

\begin{align}
	I(\alpha) & \leq \frac{c_{N-1}}{(N-1)(-\underline{\kappa})^{N/2}}\sum_{r=0}^\infty e^{-r^2}e^{(N-1)\sqrt{-\underline{\kappa}}(r+1)/\sqrt{\alpha}} \nonumber \\
	& = \frac{c_{N-1}e^{(N-1)\sqrt{-\underline{\kappa}}+\sqrt\alpha}}{(N-1)(-\underline{\kappa})^{N/2}}\sum_{r=0}^\infty e^{-r^2}e^{(N-1)\sqrt{-\underline{\kappa}}r/\sqrt\alpha} \nonumber \\
	& = \frac{c_{N-1}e^{(N-1)\sqrt{-\underline{\kappa}/\alpha}-\frac{\underline{\kappa} (N-1)^2}{\alpha}}}{(N-1)(-\underline{\kappa})^{N/2}}\sum_{r=0}^\infty e^{-\left(r-\frac{(N-1)\sqrt{-\underline{\kappa}}}{2\sqrt\alpha}\right)^2}. \nonumber
\end{align}

Now, using the inequality $\DS \sum_{r=0}^\infty e^{-(r-m)^2} \leq 5$, for any $m>0$, we obtain that 
$$I(\alpha)\leq \frac{5c_{N-1}e^{(N-1)\sqrt{-\underline{\kappa}/\alpha}-\frac{\underline{\kappa} (N-1)^2}{\alpha}}}{(N-1)(-\underline{\kappa})^{N/2}}<\infty.$$

We are now ready to prove Lemma~\ref{lemma:subG2}. By the last part of Lemma~\ref{lemmclassiquesousGauss}, it suffices to show that for sufficiently large $K>0$, it holds that $\DS \E[e^{\frac{(f(X)-\E[f(X)])^2}{2K^2}}] \leq 2$ for all $f\in\mathcal F$.

Fix $f\in\mathcal F$ and $K_0>1/\sqrt\beta$. By Jensen's inequality, $\DS \E[e^{\frac{(f(X)-\E[f(X)])^2}{2K_0^2}}] \leq \E[e^{\frac{(f(X)-f(Y))^2}{2K_0^2}}]$, where $Y$ is an independent copy of $X$. Therefore, 
\begin{align*}
	\E[e^{\frac{(f(X)-\E[f(X)])^2}{2K_0^2}}] & \leq \E[e^{\frac{d(X,Y)^2}{2K_0^2}}] \leq \E\left[e^{\frac{d(X,x_0)^2+d(Y,x_0)^2}{K_0^2}}\right] = \E\left[e^{\frac{d(X,x_0)^2}{K_0^2}}\right]^2 \\
	& =C^2I(\beta-1/K_0^2)^2 =: J
\end{align*}
which is finite as long as $K>1/\sqrt\beta$. By Hölder's inequality, it holds that for all $f\in\mathcal F$ and $K\geq K_0$, 
\begin{align*}
    \E\left[e^{\frac{(f(X)-\E[f(X)])^2}{2K^2}}\right] & \leq \left(\E\left[e^{\frac{(f(X)-\E[f(X)])^2}{2K_0^2}}\right]\right)^{K_0^2/K^2} \\
    & \leq J^{K_0^2/K^2}
\end{align*}
which goes to $1$ as $K\to\infty$. Therefore, for sufficiently large $K$ (independently of the choice of $f\in\mathcal F$, $\E[e^{\frac{(f(X)-\E[f(X)])^2}{2K^2}}]\leq 2$.


\end{document}